\documentclass[11pt]{amsart}
\usepackage{graphicx}
\usepackage{pinlabel}
\usepackage{amssymb}

\newtheorem{theorem}{Theorem}[section]
\newtheorem{proposition}[theorem]{Proposition}
\newtheorem{lemma}[theorem]{Lemma}
\newtheorem{corollary}[theorem]{Corollary}

\theoremstyle{definition}
\newtheorem{definition}[theorem]{Definition}
\newtheorem{notation}[theorem]{Notation}
\newtheorem{example}[theorem]{Example}

\newcommand{\PSL}{\operatorname{PSL}}
\newcommand{\SL}{\operatorname{SL}}

\newcommand{\Q}{\operatorname{{\mathbb Q}}}
\newcommand{\R}{\operatorname{{\mathbb R}}}
\newcommand{\Z}{\operatorname{{\mathbb Z}}}

\newcommand{\longpage}{\enlargethispage{\baselineskip}}

\begin{document}

\title[Semisimple tunnels]
{Semisimple tunnels}

\author{Sangbum Cho}
\address{Department of Mathematics Education\\
Hanyang University\\
Seoul 133-791\\
Korea}
\email{scho@hanyang.ac.kr}

\author{Darryl McCullough}
\address{Department of Mathematics\\
University of Oklahoma\\
Norman, Oklahoma 73019\\
USA}
\email{dmccullough@math.ou.edu}
\urladdr{www.math.ou.edu/$_{\widetilde{\phantom{n}}}$dmccullough/}
\thanks{The second author was supported in part by NSF grant DMS-0802424}


\date{\today}


\begin{abstract}
A knot in $S^3$ in genus-$1$  $1$-bridge position (called a
$(1,1)$-position) can be described by an element of the braid group of two
points in the torus. Our main results tell how to translate between a braid
group element and the sequence of slope invariants of the upper and lower
tunnels of the $(1,1)$-position. After using them to verify previous
calculations of the slope invariants for all tunnels of $2$-bridge knots
and $(1,1)$-tunnels of torus knots, we obtain characterizations of the
slope sequences of tunnels of $2$-bridge knots, and of a class of tunnels
we call toroidal. The main results lead to a general algorithm to calculate
the slope invariants of the upper and lower tunnels from a braid
description. The algorithm has been implemented as software, and we give
some sample computations.
\end{abstract}

\maketitle

\section*{Introduction}
\label{sec:intro}

Genus-$2$ Heegaard splittings of the exteriors of knots in $S^3$ have been
a topic of considerable interest and recent progress. Usually these are
discussed in the language of tunnels, which we will use from now on. In
particular, the term tunnel will mean a tunnel of a tunnel number $1$ knot
in~$S^3$.

A rich source of examples of tunnels are the upper and lower tunnels
associated to a knot positioned with bridge number $1$ with respect to a standard torus in  $S^3$. Traditionally this is called a
$(1,1)$-position of the knot, and the associated tunnels are called
$(1,1)$-tunnels.

In \cite{CMtree}, we laid out a theory of tunnels based on the disk complex
of the genus-$2$ handlebody. It provides a unique construction of each knot
tunnel by a sequence of ``cabling'' constructions, each determined by a
rational ``slope'' invariant (the slope invariant of the first cabling is
only defined in $\Q/\Z$). There is a second invariant, a binary sequence,
which is trivial for $(1,1)$-tunnels. Thus the sequence of slope invariants
is a complete invariant for a $(1,1)$-tunnel.

Naturally it is not very easy to calculate a complete invariant, but the
invariants are known for all the tunnels of $2$-bridge knots \cite[Section
15]{CMtree} and torus knots \cite{CMtorus}. Recently,
K. Ishihara~\cite{ishihara} has given a computational algorithm which
is effective for some examples.

There is a simple description of a $(1,1)$-position in terms of a braid of
two points in a standard torus $T$ in $S^3$: Regard the braid as two arcs in
$T\times I\subset S^3$, connect the top two points with a small trivial arc
in the ``upper'' solid torus, and similarly for the bottom two points in
the ``lower'' solid torus. Many different braids can give equivalent
$(1,1)$-positions. Some of this ambiguity is resolved by using the quotient
of the braid group by its center, which we call the reduced braid group. A
braid that produces the $(1,1)$-position is called a braid description of
it. In Section~\ref{sec:braid_group}, we will examine braid descriptions
and the reduced braid group, before detailing in
Section~\ref{sec:last_step} the first of several ``maneuvers'' involving
braids and $(1,1)$-positions. Sections~\ref{sec:tunnels}-\ref{sec:principal}
contain a review of our general theory of tunnels, focusing on the parts
needed for this paper.

Our main results are Theorems~\ref{thm:unwinding} and~\ref{thm:winding},
which allow one to pass back and forth between a braid description of a
$(1,1)$-position and the cabling slope sequence of its upper (or its lower)
tunnel. This has several applications. In Example~\ref{ex:example}, we
show how to find a braid description and use it to calculate the slope
invariants for a more-or-less random example, the knot and tunnel drawn in
Figure~10 of~\cite{CMtree}. In Section~\ref{sec:2bridge} we use braid
descriptions for the $(1,1)$-positions of all $2$-bridge knots to recover
the general calculation of slope invariants obtained in~\cite{CMtree}. We
also give a precise characterization of the slope sequences that arise from
tunnels of $2$-bridge knots. In Section~\ref{sec:torus_knots}, we use braid
descriptions for the $(1,1)$-positions of torus knots (each has a unique
$(1,1)$-position) to recover the slope invariants for their upper and lower
tunnels, first found in~\cite{CMtorus}.

A more theoretical application is given in Section~\ref{sec:toroidal},
where we show that a certain property of the sequence of slope invariants
corresponds to a $(1,1)$-position in $T\times I$ with no critical points in
either of the $S^1$-directions. We call such positions \textit{toroidal}
positions. Among the $2$-bridge knots, only the
$(2n+1,\pm2)$-torus knots admit a toroidal position.

Our final applications make the procedure of passing between braid
descriptions and slope invariants completely algorithmic. Passing from the
sequence of slope invariants to a braid description is rather easy, as
seen in Section~\ref{sec:finding_braid_words}. The other direction,
detailed in Section~\ref{sec:word_simplify}, is more difficult, since
anomalous infinite-slope cablings can arise (technically speaking, these
are not even ``cablings'') when the braid word is put into its standard
form, and one must manipulate the word to eliminate these. Both of the
algorithms, as well as the general slope calculations for $2$-bridge knot
tunnels and $(1,1)$-tunnels of torus knots, are very effective and have
been implemented in software which is available at~\cite{slopes} (other
software there finds the invariants for the ``middle'' tunnels of torus
knots). Sample calculations are given in Section~\ref{sec:computation}.

\section{Braid descriptions of $(1,1)$-positions}
\label{sec:braid_group}

In this section, we recall the 2-braid group on the torus, and its quotient
by its center. The latter, which we call the \textit{reduced braid group
  $\mathcal{B}$,} or just the \textit{braid group}, will play a central
role in our work. We will also see how an element $\omega\in \mathcal{B}$
describes a knot $K(\omega)$, and moreover a $(1,1)$-position of that
knot.

Let $T$ be a standard torus in $S^3$, bounding a solid torus $W\subset
S^3$. In our figures, $W$ usually lies above $T$.
Denoting the unit interval $[0,1]$ by $I$, fix a collar
$T\times I \subset \overline{S^3-W}$ with $T=T\times\{0\}$, and denote by
$V$ the solid torus $\overline{S^3-(W\cup T\times I)}$.

Fix a point $b\in T$, which we will refer to as the \textit{black point.}
Fix standard meridian and longitude curves $m$ and $\ell$ in $T$ such that
\begin{enumerate}
\item $m\cap \ell = b$,
\item $m$ bounds a disk in $V\cup T\times I$, and
\item $\ell$ bounds a disk in $W$.
\end{enumerate}
Choose a point $w$ in $T$ that is not in $m\cup \ell$. We will refer to $w$
as the \textit{white point.}

A braid can be described geometrically as a pair of disjoint arcs properly
embedded in $T\times I$ such that each endpoint of the arcs is one of
$b\times\{0\}$, $b\times\{1\}$, $w\times \{0\}$, or $w\times\{1\}$, and
each of the arcs meets each $T\times\{s\}$ transversely in a single
point. There is an obvious multiplication operation on the collection of
such pairs defined by ``stacking'' two pairs.

Two such pairs are equivalent if there is an isotopy $J_t$ of $T\times I$
such that
\begin{enumerate}
\item[(1)] $J_0 = id_{T\times I}$,
\item[(2)] $J_t |_{T\times \partial I} =
id_{T\times \partial I}$ for $t \in [0, 1]$,
\item[(3)] $J_t (T\times\{s\})=
T\times\{s\}$ for $s \in [0, 1]$ and $t \in [0, 1]$, and
\item[(4)] $J_1$ sends one pair to the other pair.
\end{enumerate}
The multiplication operation induces a group structure on the set of
equivalence classes, producing $\mathcal{B}_2(T)$, the classical braid
group of two points in the torus.

A presentation of $\mathcal{B}_2(T)$ is given in \cite{B} and \cite{Ta}. We
rewrite it as
\begin{gather*}\langle \delta_m, \delta_\ell, \sigma\;|\; \delta_m \sigma
\delta_m \sigma = \sigma \delta_m \sigma \delta_m, \delta_\ell \sigma \delta_\ell
\sigma = \sigma \delta_\ell \sigma \delta_\ell,\\
\delta_m^{-1} \delta_\ell \delta_m
\delta_\ell^{-1} = \sigma^2, \sigma \delta_\ell \sigma \delta_m = \delta_m \sigma
\delta_\ell \sigma^{-1} \rangle\ .
\end{gather*}
In the notation of \cite{Ta}, $\delta_m = y_1$, $\delta_\ell = x_1^{-1}$,
and $\sigma= s_1$.

As seen in Figure~\ref{fig:classical_braid},
representatives of $\delta_m$ and $\delta_\ell$ slide the black point
around $m$ and $\ell$ respectively, while keeping the white point fixed. A
representative of $\sigma$ produces a half-twist of the two strands, as
shown.
\begin{figure}
\labellist
\pinlabel $m$ [B] at 116 285
\pinlabel $\ell$ [B] at 5 226
\endlabellist
\begin{center}
\includegraphics[width=22ex]{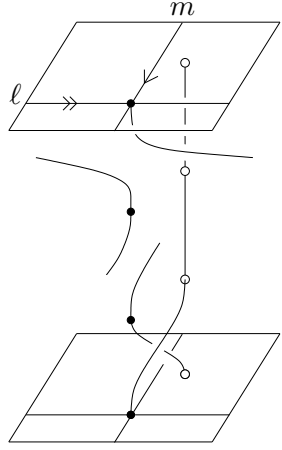}
\caption{An imbedding representing the element $\delta_\ell\delta_m\sigma$
in $\mathcal{B}_2(T)$. We read braids from the top down, so in this picture,
the solid torus $W$ lies above the upper copy of the torus, which is
$T=T\times \{0\}$.}
\label{fig:classical_braid}
\end{center}
\end{figure}

Now we weaken condition $(2)$ in the definition of equivalence of braids to
\begin{enumerate}
\item[($2'$)] $J_1 |_{T\times \partial I} = id_{T\times \partial I}$.
\end{enumerate}
That is, we do not require that each $J_t$ be the identity on $T\times \partial I$
for $t \in (0, 1)$. We call the new equivalence classes of the pairs of
arcs under this condition \textit{reduced braids}, and the group of all
reduced braids is the \textit{reduced braid group} denoted
by~$\mathcal{B}$.

The fundamental group $\pi_1(T) = \Z \times \Z$ can be regarded as a
subgroup of $\mathcal{B}_2(T)$, as the subgroup generated by $\delta_\ell
\sigma \delta_\ell\sigma$ and $\delta_m \sigma \delta_m \sigma$.
Figure~\ref{fig:braid1} illustrates a pair of arcs representing
$\delta_\ell \sigma \delta_\ell\sigma$ in $\mathcal{B}_2(T)$.  Using the
presentation of $\mathcal{B}_2(T)$, one can verify that $\pi_1(T)$ is
central in~$\mathcal B$.
\begin{figure}
\labellist
\pinlabel $m$ [B] at 116 176
\pinlabel $\ell$ [B] at 3 118
\endlabellist
\centering
\includegraphics[width=22 ex]{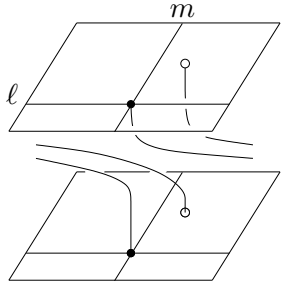}
\caption{A braid which represent $\delta_\ell \sigma \delta_\ell\sigma$. A
similar pair, winding in the $m$-direction, represents
$\delta_m\sigma\delta_m\sigma$.}
\label{fig:braid1}
\end{figure}

Weakening condition (2) to $(2')$ has the effect of making the braids
$\delta_\ell\sigma\delta_\ell\sigma$ and $\delta_m\sigma\delta_m\sigma$
trivial. On the other hand, the additional isotopies allowed by condition
$(2')$ are just products of the isotopies that move those two braids to the
trivial braid (or the reverses of these isotopies). Thus $(2')$ has the
effect of adding the two relations $\delta_m \sigma \delta_m\sigma$ and
$\delta_\ell \sigma \delta_\ell\sigma$ to the above presentation of
$\mathcal{B}_2(T)$, giving the following proposition.

\begin{proposition}
The reduced braid group $\mathcal B$ has the presentation
\begin{center}
$\langle \;\delta_m, \delta_\ell, \sigma \;|\;(\delta_m\sigma)^2 =
(\delta_\ell\sigma)^2 =1, \;\delta_m^{-1}\delta_\ell \delta_m
\delta_\ell^{-1} = \sigma^2 \; \rangle$.
\end{center}
\label{prop:braid_group_presentation}
\end{proposition}

As shown in Figure~\ref{fig:braids}(b), in $\mathcal{B}$
there are other representatives of $\delta_m$ and $\delta_\ell$ that slide
the white point backwards along loops parallel to $m$ and $\ell$
respectively.  An isotopy that moves $b\times\{0\}$ around the loop $\ell$
changes the first representative of $\delta_\ell$, that moves the black
point strand, to the second, that moves the white point strand in the other
direction. Similarly, an isotopy that moves $b\times\{0\}$ around $m$
changes the representative of $\delta_m$ that moves the black strand to one
that moves the white strand in the other direction. These correspond to the
facts that $\sigma^{-1}\delta_\ell^{-1}\sigma^{-1}=\delta_\ell$ and
similarly for~$\delta_m$.
\begin{figure}
\labellist
\pinlabel $m$ [B] at 115 365
\pinlabel $\ell$ [B] at 2 306
\pinlabel $m$ [B] at 340 365
\pinlabel $\ell$ [B] at 228 306
\pinlabel $\alpha_W$ [B] at 67 338
\pinlabel (a) [B] at 3 350
\pinlabel (b) [B] at 230 350
\pinlabel $\alpha_W$ [B] at 292 338
\pinlabel $M$ [B] at 162 49
\pinlabel $L$ [B] at 63 18
\pinlabel $M$ [B] at 387 49
\pinlabel $L$ [B] at 290 18
\pinlabel $\alpha_V$ [B] at 135 3
\pinlabel $\alpha_V$ [B] at 360 3
\endlabellist
\begin{center}
\includegraphics[width=47ex]{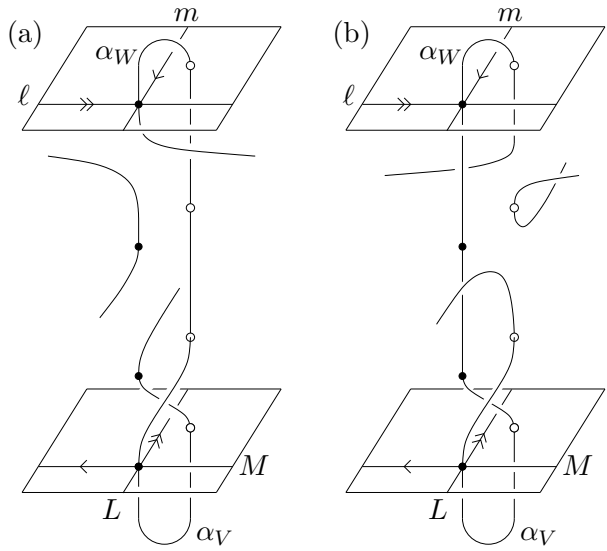}
\caption{Two imbeddings representing the element
$\delta_\ell\delta_m\sigma$ in $\mathcal{B}$. Also shown are the meridian
$M$ and longitude $L$ for the reverse braid $\sigma\delta_L\delta_M$, and
the trivial arcs $\alpha_W$ and $\alpha_V$ added to form the knot
$K(\delta_\ell\delta_m\sigma)$. It is the trivial knot.}
\label{fig:braids}
\end{center}
\end{figure}

A nice way to understand the relations in $\mathcal{B}$ is to consider the
picture of the universal cover of $T$ shown in Figure~\ref{fig:univ_cover}.
\begin{figure}
\labellist
\pinlabel $\delta_\ell$ [B] at 56 149
\pinlabel $\delta_\ell$ [B] at 127 149
\pinlabel $\delta_\ell$ [B] at 56 76
\pinlabel $\delta_\ell$ [B] at 127 76
\pinlabel $\delta_\ell$ [B] at 56 4
\pinlabel $\delta_\ell$ [B] at 127 4
\pinlabel $\delta_m$ [B] at 6 125
\pinlabel $\delta_m$ [B] at 79 125
\pinlabel $\delta_m$ [B] at 152 125
\pinlabel $\delta_m$ [B] at 6 52
\pinlabel $\delta_m$ [B] at 79 52
\pinlabel $\delta_m$ [B] at 152 52
\pinlabel $\sigma$ [B] at 39 105
\pinlabel $\sigma$ [B] at 110 105
\pinlabel $\sigma$ [B] at 39 34
\pinlabel $\sigma$ [B] at 110 34
\endlabellist
\begin{center}
\includegraphics[width=30ex]{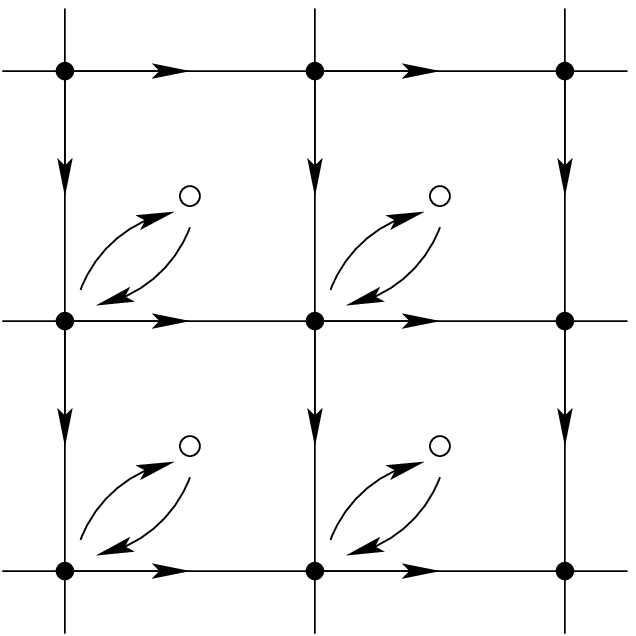}
\caption{The universal cover of $T$, illustrating the relations in
$\mathcal{B}$.}
\label{fig:univ_cover}
\end{center}
\end{figure}
In this picture, $\delta_\ell$ slides the black points one unit to the
right, $\delta_m$ slides the black points one unit downward, and $\sigma$
interchanges black and white points as indicated. The product
$\sigma\delta_\ell\sigma$ slides the white points one unit to the right, so
the effect of $\delta_\ell\sigma\delta_\ell\sigma$ is to slide both black
and white one unit to the right. The relation
$\delta_m\sigma\delta_m\sigma$ is similar, while the word
$\delta_m^{-1}\delta_\ell\delta_m\delta_\ell^{-1}$ corresponds to a
braid for which the white
points are fixed and the black points travel clockwise around the
squares in Figure~\ref{fig:univ_cover}, starting from the lower left-hand
corner, which is the effect of~$\sigma^2$.

Since we will have no further use for $\mathcal{B}_2(T)$, it is safe just
to call $\mathcal{B}$ \textit{the braid group,} and its elements
\textit{braids.}

Figure~\ref{fig:braids} also illustrates the \textit{reverse braid,} which
is obtained if one views the picture from below. The meridian and longitude
$M$ and $L$ seen from below are analogous to the meridian $m$ and longitude
$\ell$ seen from above. Note that they are interchanged, so that $\delta_m$
seen from below is $\delta_L$ and $\delta_\ell$ seen from below is
$\delta_M$. Although both have the reversed orientation, a braid
$\delta_m$ or $\delta_\ell$ seen from below has the point moving in
reversed time, so $\delta_m$ becomes $\delta_L$ and $\delta_\ell$ becomes
$\delta_M$. On the other hand, $\sigma$ seen from below still looks like
$\sigma$. Thus the reverse braid has $\delta_m$ replaced by $\delta_L$ and
$\delta_\ell$ replaced by $\delta_M$, with $\sigma$ unchanged, and the
order of the letters reversed.

Figure~\ref{fig:braids} also illustrates the \textit{knot described by the
  braid $\omega$.} One simply attaches the two standard arcs $\alpha_V$ and
$\alpha_W$ at the bottom and top.  In Figure~\ref{fig:braids}, $\alpha_V$
and $\alpha_W$ are obtained from arcs $\alpha \times \{1\}$ and $\alpha
\times \{0\}$ respectively, pushing the former slightly into $V$ and the
latter into $W$, where $\alpha$ is a fixed arc in $T$ connecting $b$ and
$w$, and meeting $m \cup \ell$ only in~$b$. There are four isotopy classes
of such arcs, and we select $\alpha$ to lie in the isotopy class that
at $b$ leaves $m$ in the direction of the positive orientation on $\ell$ (using
the orientation shown in of Figure~\ref{fig:braids}) and leaves $\ell$ in
the direction of the negative orienation on $m$. The lifts of $\alpha$ to
the universal covering space of $T$ shown in Figure~\ref{fig:univ_cover}
leave the preimages of $b$ from the lower-left hand corner of each square.

The resulting knot $K(\omega)$ is well-defined, indeed equivalent reduced
braids describe knots that are in $(1,1)$-position with respect to $T$ and
are $(1,1)$-isotopic (that is, isotopic by an isotopy of $S^3$ preserving
$T$ at all times). The notation $K(\omega)$ implicitly includes this
well-defined $(1,1)$-position.  We say that $K(\omega)$ is \textit{in braid
position,} and that the element $\omega$ is a \textit{braid description}
of the knot and its $(1,1)$-position (with respect to the fixed arc
$\alpha$). A braid and its reverse braid describe isotopic knots, but the
two $(1,1)$-positions have upper and lower tunnels interchanged.

Observe that $K(\omega)$ and $K(\omega\delta_m)$ are $(1,1)$-isotopic, by
an isotopy that pushes $K(\omega)$ across the core circle of the solid
torus $V$. Similarly, $K(\omega)$ is $(1,1)$-isotopic to
$K(\delta_\ell\omega)$, and $K(\omega)$ to both $K(\omega\sigma)$ and
$K(\sigma\omega)$. In general, if $W_0(\delta_\ell,\sigma)$ is a word in
the letters $\delta_\ell$ and $\sigma$, and $W_1(\delta_m,\sigma)$ is a
word in $\delta_m$ and $\sigma$, then $K(W_0(\delta_\ell,\sigma)\omega
W_1(\delta_m,\sigma))$ and $K(\omega)$ are $(1,1)$-isotopic. For example, the
knot $K(\delta_\ell\delta_m\sigma)$ in Figure~\ref{fig:braids} is
$(1,1)$-isotopic to $K(1)$, hence is trivial.

\begin{notation}\label{not:double_coset}
Let $\langle \delta_\ell,\sigma\rangle$ denote the subgroup of
$\mathcal{B}$ generated by $\delta_\ell$ and $\sigma$, and similarly for
$\langle \delta_m,\sigma\rangle$. We will write $\omega_1\sim \omega_2$ to
mean that $\omega_1$ and $\omega_2$ represent the same double coset of the
form $\langle \delta_\ell,\sigma\rangle \omega_1\langle
\delta_m,\sigma\rangle$, and consequently are braid descriptions of
the same $(1,1)$-position.
\end{notation}

\section{Tunnels in standard position}
\label{sec:last_step}

\begin{figure}
\labellist
\large \hair 2pt
\pinlabel $\alpha$ [B] at 85 114
\pinlabel $\alpha_W$ [B] at 486 206
\pinlabel (a) [B] at 80 -2
\pinlabel (b) [B] at 284 -2
\pinlabel (c) [B] at 485 -2
\endlabellist
\begin{center}
\includegraphics[width=64ex]{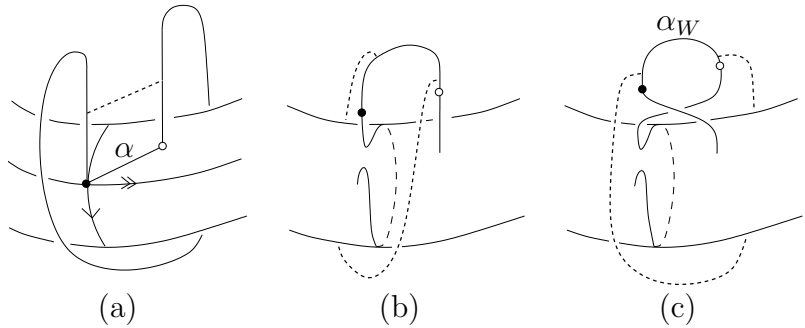}
\caption{Introduction of an initial $\sigma^{-1}\delta_m^{-1}$ in order to
reposition a tunnel into standard position.}
\label{fig:last_step}
\end{center}
\end{figure}
In this section, we will introduce a basic maneuver.
Figure~\ref{fig:last_step}(a) shows the torus $T=T\times\{0\}$, the portion of
a $(1,1)$-knot that lies in $W$, and a tunnel arc for the ``upper''
(1,1)-tunnel of the knot. The segment on $T$ connecting black and white
vertices is the fixed arc $\alpha$ used
to define the $(1, 1)$-knot from a braid
description as in Section~\ref{sec:braid_group}.
The knot is not in braid position,
because it does not meet $W$ in an arc parallel to
$\alpha$. There is an isotopy of the knot, preserving $(1,1)$-position,
that changes the placement in Figure~\ref{fig:last_step}(a) to the one in
Figure~\ref{fig:last_step}(b).  The arc of the knot in $W$ moves to a
different arc, and a braid letter $\delta_m^{-1}$ appears just inside the
solid torus. The tunnel arc is stretched out to wind around the core of
$W$. A further $(1,1)$-isotopy changes Figure~\ref{fig:last_step}(b) to
Figure~\ref{fig:last_step}(c), introducing another braid letter
$\sigma^{-1}$. If we push the knot down into $T\times I$ until its
intersection with $W$ is parallel to $\alpha$ (or equivalently enlarge $T\times I$ to
include the portion of the knot that lies below the black and white
points), then the knot is now in braid position (assuming that its
original intersection with $V\cup T\times I$ is consistent with braid position).

When the upper tunnel and the portion of $K$ in $W$ consist of a tunnel arc
and a trivial arc as in the portion above the black and white points in
Figure~\ref{fig:last_step}(c), we say that $K$ and the upper tunnel (arc)
are in \textit{standard position.}  Thus the effect of the maneuver just
described is to move a tunnel arc and $K$ that appear as in
Figure~\ref{fig:last_step}(a) so that $K$ and the tunnel are in standard
position, while changing the braid represented by the portion of $K$ in
$T\times I$ by premultiplication by $\sigma^{-1}\delta_m^{-1}$. Since
$\sigma^{-1}\delta_m^{-1}=(\delta_m\sigma)^{-1}=\delta_m\sigma$ in
$\mathcal{B}$, this is equivalent to premultiplication by $\delta_m\sigma$.

\section{Tunnels as disks}
\label{sec:tunnels}

Our previous articles \cite{CMtree,CMtorus,CMgiantsteps,CMdepth} give the
details of our general theory of knot tunnels. Of these, \cite{CMtree} is
the most complete, while of the shorter summaries in the other papers, the
material in \cite{CMtorus} is the closest to our present needs. For
convenience of the reader, we will provide in this and the next three
sections a review adapted to our work here. It also establishes notation that
is used in the rest of the paper.

This section gives a brief overview of the theory in~\cite{CMtree}. Fix a
``standard'' genus-2 unknotted handlebody $H$ in $S^3$. Regard a tunnel of
$K$ as a $1$-handle attached to a neighborhood of $K$ to obtain an
unknotted genus-$2$ handlebody. Moving this handlebody to $H$ by an isotopy
of $S^3$, a cocore disk for the $1$-handle moves to a nonseparating disk in
$H$. The indeterminacy due to the choice of isotopy is exactly the Goeritz
group, which is the group of path components of the space of
orientation-preserving homeomorphisms of $S^3$ that take $H$ onto
$H$. Consequently, the collection of all tunnels of all tunnel number~$1$
knots, up to orientation-preserving homeomorphism, corresponds to the
orbits of nonseparating disks in $H$ under the action of the Goeritz group.
Indeed, it is convenient for us to \textit{define} a tunnel to be an orbit
of nonseparating disks in $H$ under the action of the Goeritz group.

Work of Scharlemann, Akbas, and Cho~\cite{Akbas,Cho,ScharlemannTree} gives
a very good understanding of the way that the Goeritz group acts on the
disks in $H$. As detailed in~\cite{CMtree}, the orbits, i.e.~the tunnels,
can be arranged in a treelike structure which encodes much of the
topological structure of tunnel number~$1$ knots and their tunnels.

When a nonseparating disk $\tau \subset H$ is regarded as a representative
of a tunnel, or simply a tunnel itself, the corresponding knot is a core
circle of the solid torus that results from cutting $H$ along $\tau$. This
knot is denoted by~$K_\tau$. For example, in the handlebody shown in
Figure~\ref{fig:slopedisks}, a
core circle of the solid torus cut off by the middle disk is a trefoil
knot.

A disk $\tau$ in $H$ is called \textit{primitive} if there is a disk
$\tau'$ in $\overline{S^3-H}$ such that $\partial \tau$ and $\partial
\tau'$ cross in one point in $\partial H$. Equivalently, $K_\tau$ is the
trivial knot in~$S^3$. All primitive disks are equivalent under the action
of the Goeritz group. This equivalence class is the unique tunnel of the
trivial knot.

A \textit{primitive pair} is an isotopy class of two disjoint nonisotopic
primitive disks in $H$. A \textit{primitive triple} is defined
similarly. All primitive pairs are equivalent under the Goeritz group, as
are all primitive triples.

It is important to understand that a triple of nonseparating disks in $H$
corresponds to an isotopy class of $\theta$-curves in $H$, specifically,
the $\theta$-curve whose arcs are ``dual'' to the three disks--- each arc
cuts across exactly one of the three disks once, each disk meets exactly
one of the three arcs, and each of the balls obtained by cutting $H$ along
the union of the three disks deformation retracts to the portion of the
$\theta$-curve that it contains. This $\theta$-curve is ``unknotted'' in
$S^3$, that is, the closure of the complement of a regular neighborhood is
a genus-$2$ handlebody. Thus an orbit of such triples under the Goeritz
group corresponds to an isotopy class in $S^3$ of unknotted
$\theta$-curves.

\section{Slope disks and cabling arcs}
\label{sec:slopedisks}

This section gives the definitions needed for computing the slope
invariants that will be discussed in Section~\ref{sec:cabling}.
Fix a pair of disjoint nonseparating disks $\lambda$ and $\rho$ (for ``left'' and
``right'') in the standard unknotted handlebody $H$ in $S^3$, as shown
abstractly in Figure~\ref{fig:slopedisks}. The pair $\{\lambda, \rho\}$ is
arbitrary, so in the true picture in $H$ in $S^3$, they will typically look
a great deal more complicated than the pair shown in
Figure~\ref{fig:slopedisks}. Let $N$ be a regular neighborhood of
$\lambda\cup \rho$ and let $B$ be the closure of $H-N$. The frontier of $B$
in $H$ consists of four disks which appear vertical in
Figure~\ref{fig:slopedisks}. Denote this frontier by $F$, and let $\Sigma$
be $B\cap \partial H$, a sphere with four holes.
\begin{figure}
\labellist
\small \hair 2pt
\pinlabel $\lambda$ [B] at -18 136
\pinlabel $\rho$ [B] at 595 136
\endlabellist
\begin{center}
\includegraphics[width=65 ex]{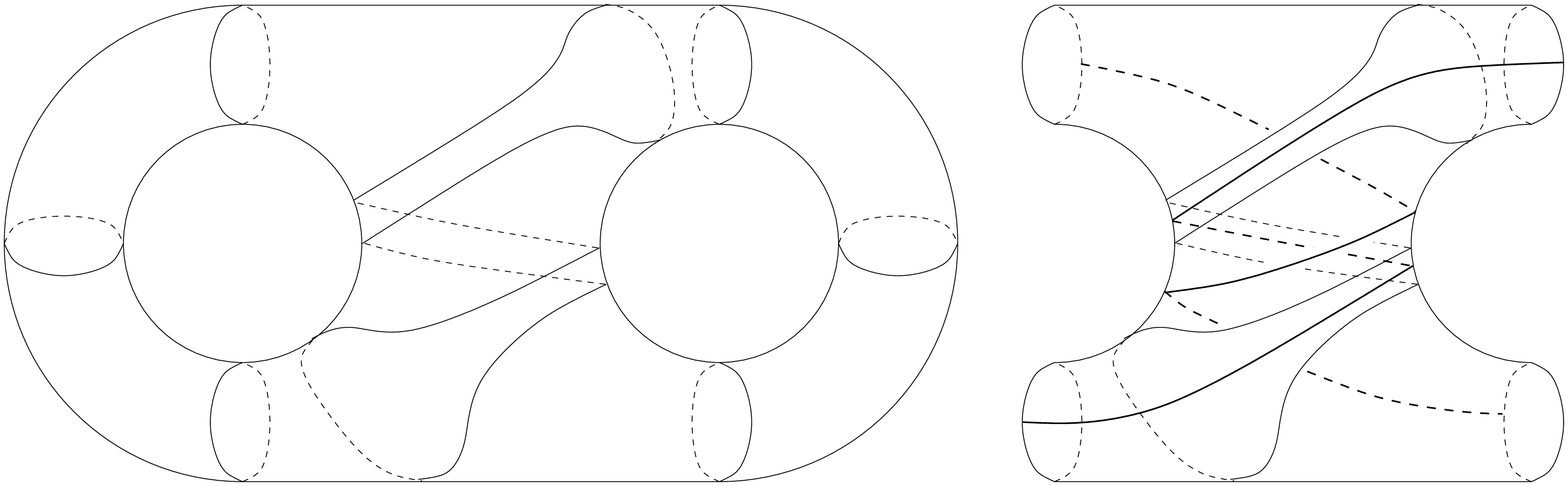}
\caption{A slope disk of $\{\lambda,\rho\}$, and a pair of its cabling arcs
contained in $B$.}
\label{fig:slopedisks}
\end{center}
\end{figure}

A \textit{slope disk for $\{\lambda,\rho\}$} is an essential disk in $H$,
possibly separating, which is contained in $B-F$ and is not isotopic to any
component of~$F$. Any loop in $\Sigma$ that is not homotopic into $\partial
\Sigma$ is the boundary of a unique slope disk.  (Throughout our work,
``unique'' means unique up to isotopy in an appropriate sense.) If two
slope disks are isotopic in $H$, then they are isotopic in~$B$. The
boundary of a slope disk always separates $\Sigma$ into two pairs of pants.

An arc in $\Sigma$ whose endpoints lie in two different boundary circles of
$\Sigma$ is called a \textit{cabling arc.}  Figure~\ref{fig:slopedisks} shows a
pair of cabling arcs disjoint from a slope disk. A slope disk is disjoint
from a unique pair of cabling arcs, and each cabling arc determines a
unique slope disk.

Each choice of nonseparating slope disk for a pair $\mu=\{\lambda,\rho\}$
determines a correspondence between~$\Q\cup\{\infty\}$ and the set of
isotopy classes of slope disks of $\mu$, as follows. Fixing a nonseparating
slope disk $\tau$ for $\mu$, write $(\mu;\tau)$ for the ordered pair
consisting of $\mu$ and~$\tau$.
\begin{definition} A \textit{perpendicular disk} for $(\mu;\tau)$
is a disk $\tau^\perp$, with the following properties:
\begin{enumerate}
\item $\tau^\perp$ is a slope disk for $\mu$.
\item $\tau$ and $\tau^\perp$ intersect transversely in one arc.
\item $\tau^\perp$ separates $H$.
\end{enumerate}
\end{definition}
There are infinitely many choices for $\tau^\perp$, but because $H\subset
S^3$ there is a natural way to choose a particular one, which we call
$\tau^0$. It is illustrated in Figure~\ref{fig:slope_coords}. To construct
it, start with any perpendicular disk and change it by Dehn twists of $H$
about $\tau$ until the core circles of the complementary solid tori have
linking number~$0$ in~$S^3$.
\begin{figure}
\labellist
\small \hair 2pt
\pinlabel $\lambda^+$ [B] at 148 298
\pinlabel $\rho^+$ [B] at 437 300
\pinlabel $\lambda^-$ [B] at 149 -25
\pinlabel $\rho^-$ [B] at 433 -21
\pinlabel $\lambda$ [B] at -15 140
\pinlabel $\tau$ [B] at 376 141
\pinlabel $\rho$ [B] at 594 140
\pinlabel $K_\rho$ [B] at 210 222
\pinlabel $K_\lambda$ [B] at 369 222
\pinlabel $\tau^0$ [B] at 291 298
\endlabellist
\begin{center}
\includegraphics[width=45 ex]{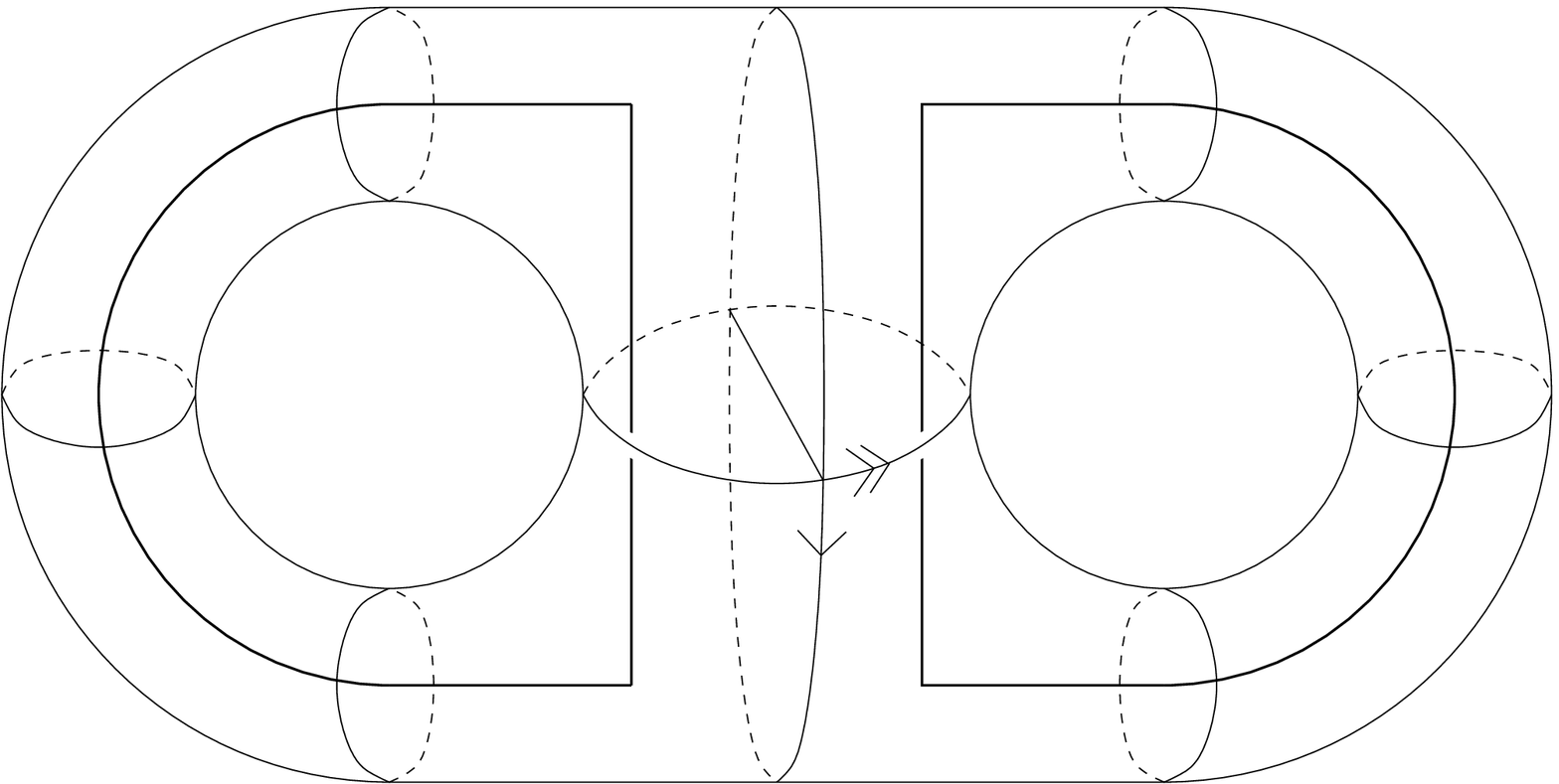}
\caption{The slope-zero perpendicular disk $\tau^0$. It is chosen so that
$K_\lambda$ and $K_\rho$ have linking number~$0$.}
\label{fig:slope_coords}
\end{center}
\end{figure}

For calculations, it is convenient to draw the picture as in
Figure~\ref{fig:slope_coords}, and orient the boundaries of $\tau$ and
$\tau^0$ so that the orientation of $\tau^0$ (the ``$x$-axis''), followed
by the orientation of $\tau$ (the ``$y$-axis''), followed by the outward
normal of $H$, is a right-hand orientation of $S^3$. At the other
intersection point, these give the left-hand orientation. The coordinates
will be unaffected by changing which of the disks in $\{\lambda,\rho\}$ is
called $\lambda$ and which is~$\rho$.

Let $\widetilde{\Sigma}$ be the covering space of $\Sigma$ such that:
\begin{enumerate}
\item $\widetilde{\Sigma}$ is the plane with an open disk of radius $1/8$
removed at each point with coordinates in $\Z\times \Z+(\frac12,\frac12)$.
\item The components of the preimage of $\tau$ are the vertical lines
with integer $x$-coordinate.
\item The components of the preimage of $\tau^0$ are the horizontal lines
with integer $y$-coordinate.
\end{enumerate}
\noindent Figure~\ref{fig:covering} shows a picture of $\widetilde{\Sigma}$
and a fundamental domain for the action of its group of covering
transformations, which is the orientation-preserving subgroup of the group
generated by reflections in the half-integer lattice lines (that pass
through the centers of the missing disks). Each circle of
$\partial\widetilde{\Sigma}$ double covers a circle of~$\partial \Sigma$.
\begin{figure}
\labellist
\small \hair 2pt
\pinlabel $\lambda^+$ [B] at 65 66
\pinlabel $\lambda^-$ [B] at 138 66
\pinlabel $\lambda^+$ [B] at 209 66
\pinlabel $\lambda^-$ [B] at 281 66
\pinlabel $\rho^+$ [B] at 65 142
\pinlabel $\rho^-$ [B] at 138 142
\pinlabel $\rho^+$ [B] at 209 142
\pinlabel $\rho^-$ [B] at 281 142
\pinlabel $\lambda^+$ [B] at 65 210
\pinlabel $\lambda^-$ [B] at 138 210
\pinlabel $\lambda^+$ [B] at 209 210
\pinlabel $\lambda^-$ [B] at 281 210
\pinlabel $\rho^+$ [B] at 65 286
\pinlabel $\rho^-$ [B] at 138 286
\pinlabel $\rho^+$ [B] at 209 286
\pinlabel $\rho^-$ [B] at 281 286
\pinlabel $\lambda^+$ [B] at 420 0
\pinlabel $\lambda^-$ [B] at 740 0
\pinlabel $\rho^+$ [B] at 420 310
\pinlabel $\rho^-$ [B] at 745 309
\pinlabel $\tau^0$ [B] at 738 158
\pinlabel $\tau$ [B] at 578 244
\endlabellist
\begin{center}
\includegraphics[width=\textwidth]{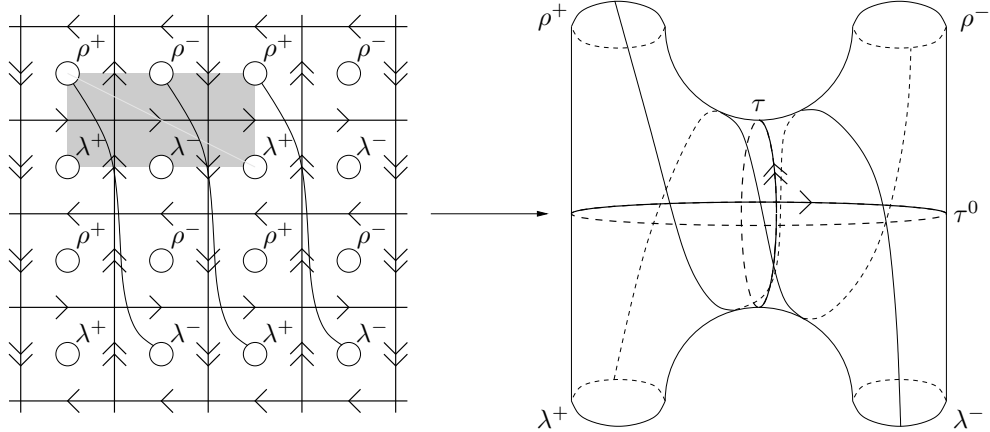}
\caption{The covering space $\widetilde{\Sigma}\to\Sigma$, and some lifts
of the pair of $[1,-3]$-cabling arcs from
Figure~\ref{fig:slopedisks}. The shaded region is a fundamental domain.}
\label{fig:covering}
\end{center}
\end{figure}

Each lift of a cabling arc $\alpha$ of $\Sigma$ to $\widetilde{\Sigma}$
runs from a boundary circle of $\widetilde{\Sigma}$ to one of its
translates by a vector $(p,q)$ of signed integers, defined up to
multiplication by the scalar $-1$. In this way $\alpha$ receives a
\textit{slope pair} $[p,q]=\{(p,q),(-p,-q)\}$, and is called a
\textit{$[p,q]$-cabling arc.} The corresponding slope disk is assigned the
slope pair $[p,q]$ as well, and can be called a \textit{$[p,q]$-slope
  disk.} The cabling arcs in Figure~\ref{fig:covering} are $[1,-3]$-cabling
arcs. A corresponding $[1,-3]$-slope disk is the one shown in
Figure~\ref{fig:slopedisks}.

An important observation is that a $[p,q]$-slope disk is nonseparating in
$H$ if and only if $q$ is odd. Both happen exactly when a corresponding
cabling arc has one endpoint in $\lambda^+$ or $\lambda^-$ and the other in
$\rho^+$ or~$\rho^-$.

\begin{definition} Let $\lambda$, $\rho$, and $\tau$ be as above, and let
$\mu=\{\lambda,\rho\}$. The \textit{$(\mu;\tau)$-slope} of a $[p,q]$-slope
disk or cabling arc is~$q/p\in \Q\cup\{\infty\}$.
\end{definition}
\noindent The $(\mu;\tau)$-slope of $\tau^0$ is $0$, the
$(\mu;\tau)$-slope of $\tau$ is~$\infty$, and the $(\mu;\tau)$-slope of a
slope disk that looks like the one in Figure~\ref{fig:slopedisks} is
$-3$. The $(\mu;\tau)$-slope can also be called the
$\{\tau,\tau^0\}$-slope, when the choice of $\mu$ is clear.

Slope disks for a primitive pair are called \textit{simple} disks, and are
handled in a special way. Rather than using a particular choice of $\tau$
from the context, one chooses $\tau$ to be some third primitive disk.
Altering this choice can change $[p,q]$ to any $[p+nq,q]$, but the quotient
$p/q$ is well-defined as an element of $\Q/\Z\cup\{\infty\}$. This element
$[p/q]$ is called the \textit{simple slope} of the slope disk. For example,
if $\mu$ is a primitive pair, the simple slope of the disk from
Figure~\ref{fig:slopedisks} is $[2/3]$. The simple slope of a slope disk is
$[0]$ exactly when the slope disk is itself primitive. Simple disks have
the same simple slope exactly when they are equivalent by an element of the
Goeritz group.

\section{The cabling construction}
\label{sec:cabling}

In a sentence, the cabling construction (sometimes just called a
\textit{cabling}) is to ``Think of the union of $K$ and the tunnel arc as a
$\theta$-curve, and rationally tangle the ends of the tunnel arc and one of
the arcs of $K$ in a neighborhood of the other arc of $K$.''  We sometimes
call this ``swap and tangle,'' since one of the arcs in the knot is
exchanged for the tunnel arc, then the ends of other arc of the knot and
the tunnel arc are connected by a rational tangle.

Figure~\ref{fig:cabling} illustrates two cablings, one starting with the
trivial knot and obtaining the trefoil, then another starting with the
tunnel of the trefoil.
\begin{figure}
\labellist
\small \hair 2pt
\pinlabel $\pi_0$ [B] at -7 178
\pinlabel $\pi$ [B] at 65 178
\pinlabel $\pi_1$ [B] at 120 178
\pinlabel $\pi_0$ [B] at 177 178
\pinlabel $\tau_0$ [B] at 239 227
\pinlabel $\pi_1$ [B] at 304 177
\pinlabel $\pi_1$ [B] at 66 121
\pinlabel $\tau_0$ [B] at 127 93
\pinlabel $\pi_0$ [B] at 93 39
\pinlabel $\tau_1$ [B] at 226 54
\pinlabel $\pi_0$ [B] at 250 32
\pinlabel $\tau_0$ [B] at 290 99
\endlabellist
\begin{center}
\includegraphics[width=50 ex]{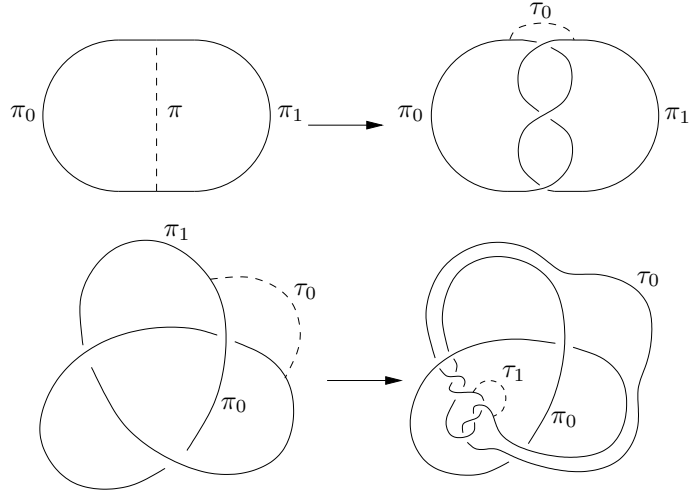}
\caption{Examples of the cabling construction.}
\label{fig:cabling}
\end{center}
\end{figure}

More precisely, begin with a triple $\{\lambda,\rho,\tau\}$, regarded as a
pair $\mu=\{\lambda,\rho\}$ with a slope disk $\tau$ which represents a
tunnel. Choose one of the disks in $\{\lambda,\rho\}$, say $\lambda$, and a
nonseparating slope disk $\tau'$ of the pair $\{\lambda,\tau\}$,
\textit{other than $\rho$.} This is a cabling operation producing the
tunnel $\tau'$ from $\tau$. In terms of the ``swap and tangle'' description
of a cabling, $\lambda$ is dual to the arc of $K_\tau$ that is retained,
and the slope disk $\tau'$ determines a pair of cabling arcs that form the
rational tangle that replaces the arc of $K_\tau$ dual to~$\rho$.

Provided that $\{\lambda,\rho,\tau\}$ was not a primitive triple, we define
the \textit{slope} of this cabling operation to be the
$(\{\lambda,\tau\};\rho)$-slope of~$\tau'$.  When $\{\lambda,\rho,\tau\}$
is primitive, the cabling construction starts with the tunnel of the
trivial knot and produces an upper or lower tunnel of a $2$-bridge knot,
unless $\tau'$ is primitive, in which case it is again the tunnel of the
trivial knot and the cabling is called \textit{trivial.} The slope of a
cabling starting with a primitive triple is defined to be the simple slope
of $\tau'$. The cabling is trivial when the simple slope is~$[0]$.

Since tunnel disks for knot tunnels are nonseparating, the slope invariant
of a cabling construction producing a knot tunnel is of the form $q/p$ with
$q$ odd (or $[p/q]$ with $q$ odd, for a simple slope). In this paper, all
cablings will use nonseparating disks and produce knots. In general, a
cabling construction can also use a separating disk as $\tau'$, which will
produce a tunnel of a tunnel number $1$ link, and no further cabling is
then possible. The slope invariant of such a cabling is defined in the same
way, and has $q$ even.

A nontrivial tunnel $\tau_0$ produced from the tunnel of the trivial knot
by a single cabling construction is called a \textit{simple} tunnel. These
are the well-known ``upper and lower'' tunnels of $2$-bridge knots. Not
surprisingly, the simple slope $m_0$ is a version of the standard rational
parameter that classifies the $2$-bridge knot~$K_{\tau_0}$.

A tunnel is called \textit{semisimple} if it is disjoint from a primitive
disk, but not from any primitive pair. The simple and semisimple tunnels
are exactly the $(1,1)$-tunnels, that is, the upper and lower tunnels of
knots in $1$-bridge position with respect to a standard torus of~$S^3$.
A tunnel is called \textit{regular} if it is neither primitive, simple, or
semisimple.

\section{The tunnel invariants and the principal vertex}
\label{sec:principal}

Theorem~13.2 of~\cite{CMtree} shows that every tunnel of every tunnel
number~$1$ knot can be obtained by a uniquely determined sequence of
cabling constructions. The associated cabling slopes form a
sequence
\[ m_0,\;m_1,\;\cdots\;,\;m_n = [p_0/q_0],\;q_1/p_1,\;\cdots\;,\;q_n/p_n\]
where $m_0\in\Q/\Z$ and each $q_i$ is odd, called the sequence of
\textit{slope invariants} of the tunnel, or just its \textit{slope sequence.}

The unique sequence of cabling constructions producing a tunnel $\tau$
begins with a primitive triple $\{\lambda_{-1},\rho_{-1},\tau_{-1}\}$,
where $\tau_{-1}$ is regarded as the tunnel of the trivial knot. The
cabling constructions produce triples $\{\lambda_i,\rho_i,\tau_i\}$ for
$0\leq i\leq n$, each $\{\lambda_i,\rho_i,\tau_i\}$ is either
$\{\lambda_{i-1},\tau_{i-1},\tau_i\}$ or $\{\tau_{i-1},
\rho_{i-1},\tau_i\}$. The triple $\{\lambda_n,\rho_n,\tau_n\}$ is called
the \textit{principal vertex} of $\tau$. It is called a vertex because it
corresponds to a vertex in the ``tree of knot tunnels''
of~\cite{CMtree}). The uniquenss of the sequence of cabling constructions
producing $\tau$ is really just the fact that in this tree there is a
unique arc between any two vertices--- in this case the unique vertex
corresponding to the primitive triple and the nearest vertex containing
$\tau$--- and each cabling construction corresponds to a step along this
arc.

As we noted in Section~\ref{sec:tunnels}, the principal vertex is dual to a
specific isotopy class of $\theta$-curves in $S^3$. The arc of this
$\theta$-curve that is dual to the tunnel disk furnishes a canonical tunnel
arc representing the tunnel (and the other two arcs form the knot). Indeed,
each of the triples $\{\lambda_i,\rho_i,\tau_i\}$ determines the canonical
tunnel arc representative of the tunnel $\tau_i$.  Geometrically, each
cabling construction along the way produces the canonical tunnel arc of the
resulting tunnel.

There is a second set of invariants associated to a tunnel. Each $m_i$ is
the slope of a cabling that begins with a triple of disks
$\{\lambda_{i-1},\rho_{i-1},\tau_{i-1}\}$ and finishes with
$\{\lambda_i,\rho_i,\tau_i\}$. For $i\geq 2$, put $s_i=1$ if
$\{\lambda_i,\rho_i,\tau_i\}=\{\tau_{i-2},\tau_{i-1},\tau_i\}$,
and $s_i=0$ otherwise. In terms of the swap-and-tangle construction, the
invariant $s_i$ is $1$ exactly when the rational tangle replaces the arc of
the knot that was retained by the previous cabling (for $i=1$, the choice
does not matter, as there is an element of the Goeritz group that preserves
$\tau_0$ and interchanges $\lambda_0$ and $\rho_0$).

A tunnel is simple or semisimple if and only if all $s_i=0$. The reason is
that both conditions characterize cabling sequences in which one of the
original primitive disks is retained in every cabling; this corresponds to
the fact that there exists a tunnel arc for the tunnel (the canonical
tunnel arc, in fact) whose union with one of the arcs of the knot is
unknotted.

\section{Standard tangles}
\label{sec:standard_tangles}

As we have seen, the cabling construction involves the replacement of a
portion of a knot by a rational tangle in the ball $B$. In our later slope
calculations, the rational tangle is positioned in $B$ according to the
picture seen in Figure~\ref{fig:standard_tangle}, which we call a standard
tangle. Notice that a standard tangle is isotopic (keeping endpoints in the
frontier of $B$) to a unique cabling arc. So it
makes sense to speak of the slope of a standard tangle. The next
proposition gives a simple expression for this slope.

\begin{figure}
\labellist
\pinlabel $b_n$ [B] at 120 90
\pinlabel $b_{n-1}$ [B] at 313 90
\pinlabel $b_1$ [B] at 580 90
\pinlabel $a_n$ [B] at 215 50
\pinlabel $a_1$ [B] at 675 50
\endlabellist
\begin{center}
\includegraphics[width=62 ex]{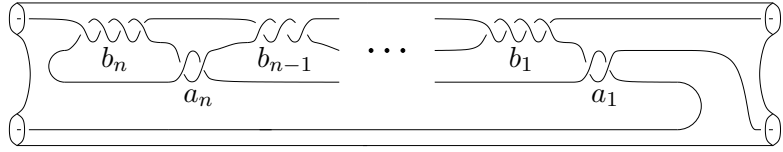}
\caption{A standard tangle of type $(a_1,b_1,\ldots, a_n,b_n)$. For the
$a_i$, each left-hand half twist counts as $+1$, while for the $b_i$, each
right-hand half twist does.}
\label{fig:standard_tangle}
\end{center}
\end{figure}
\begin{proposition} In the coordinates coming from the pair
  $\{\tau,\tau^0\}$ shown in the first drawing in
  Figure~\ref{fig:2bridge_cable}, the standard tangle of type
  $(a_1,b_1,\ldots, a_n,b_n)$ has slope given by the continued fraction
  $[a_1,b_1,\ldots, a_n,b_n]$.\par
\label{prop:slope_of_standard_tangle}
\end{proposition}

\begin{proof} We write
$U=\begin{bmatrix}1&1\\0&1\end{bmatrix}$ and
$L=\begin{bmatrix}1&0\\1&1\end{bmatrix}$, and
refer to Figure~\ref{fig:2bridge_cable}.
\begin{figure}
\labellist
\pinlabel $\tau^0$ [B] at 158 315
\pinlabel $\tau$ [B] at 230 274
\pinlabel $b_n$ [B] at 703 287
\pinlabel $b_n$ [B] at 540 90
\pinlabel $a_n$ [B] at 633 53
\endlabellist
\begin{center}
\includegraphics[width=66ex]{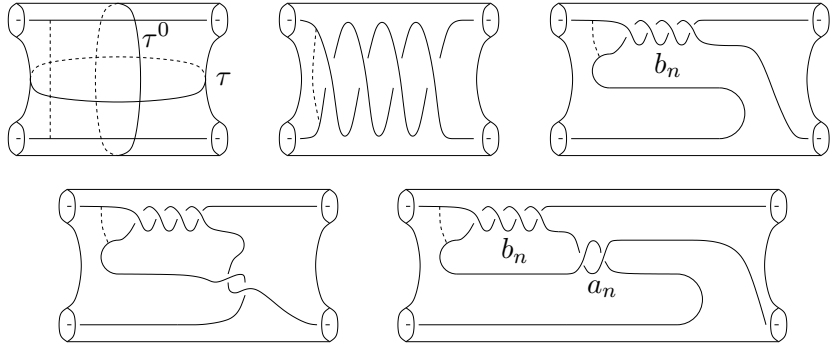}
\caption{Calculating the slope of a standard tangle.}
\label{fig:2bridge_cable}
\end{center}
\end{figure}
The pair of arcs in the first picture of $B$
has slope $\infty$, and performing $b_n$ left-hand half twists of the right
half of $B$ produces the pair in the second picture, which
has slope coordinates $[b_n,1]$. Regarding a pair of slope
coordinates $[p,q]$ as a column vector $\begin{bmatrix} q \\
p\end{bmatrix}$, this is expressed algebraically by the calculation
\[ L^{b_n}\,\begin{bmatrix} 1 \\ 0\end{bmatrix}
= \begin{bmatrix}1&0\\b_n&1\end{bmatrix}\,
\begin{bmatrix} 1 \\  0\end{bmatrix}
= \begin{bmatrix} 1 \\ b_n\end{bmatrix} \ ,\] in which the resulting column
vector gives the slope coordinates of the resulting cable.
Next, we perform $a_n$ right-hand half twists of the bottom half
of $B$. As seen in the third, fourth, and fifth pictures of
Figure~\ref{fig:2bridge_cable}, this moves the arcs to the standard
tangle of type $(a_n,b_n)$. The effect of a right-hand half twist on
slope coordinates is to send $[p, q]$ to $[p, q + p]$, which is the
effect of multiplication by $U$. So the resulting slope coordinates from
$a_n$ twists are
\[ U^{a_n}L^{b_n}\,\begin{bmatrix} 1 \\ 0\end{bmatrix}
= \begin{bmatrix}1&a_n\\0&1\end{bmatrix}\,
\begin{bmatrix} 1 \\  b_n\end{bmatrix}
= \begin{bmatrix} 1 + a_nb_n \\ b_n\end{bmatrix} \ ,\]
and the slope is $a_n+1/b_n=[a_n,b_n]$. An inductive calculation shows
that
\[ U^{a_1}L^{b_1}\cdots U^{a_n}L^{b_n}\,\begin{bmatrix} 1 \\ 0\end{bmatrix}
= \begin{bmatrix} q \\  p\end{bmatrix}
\]
where $q/p=[a_1,b_1,\ldots,a_n,b_n]$ (see \cite[Lemma 14.3]{CMtree}),
verifying the proposition.
\end{proof}

\section{Unwinding rational tangles}
\label{sec:unwinding}

In this section, we introduce two isotopy maneuvers similar to the one in
Section~\ref{sec:last_step}. They reposition a knot that is produced by a
cabling construction on a knot and tunnel in the standard position detailed
in Section~\ref{sec:last_step}. This leads to our first main result,
Theorem~\ref{thm:unwinding}, which tells how performing a cabling
construction on an upper tunnel in standard position changes a braid
description of the $(1,1)$-position.

\begin{figure}
\labellist
\pinlabel $2a_1$ [B] at 310 412
\pinlabel $b_1$ [B] at 253 326
\pinlabel {\large(a)} [B] at 273 -20
\pinlabel {\large(b)} [B] at 943 -20
\endlabellist
\begin{center}
\includegraphics[width=65ex]{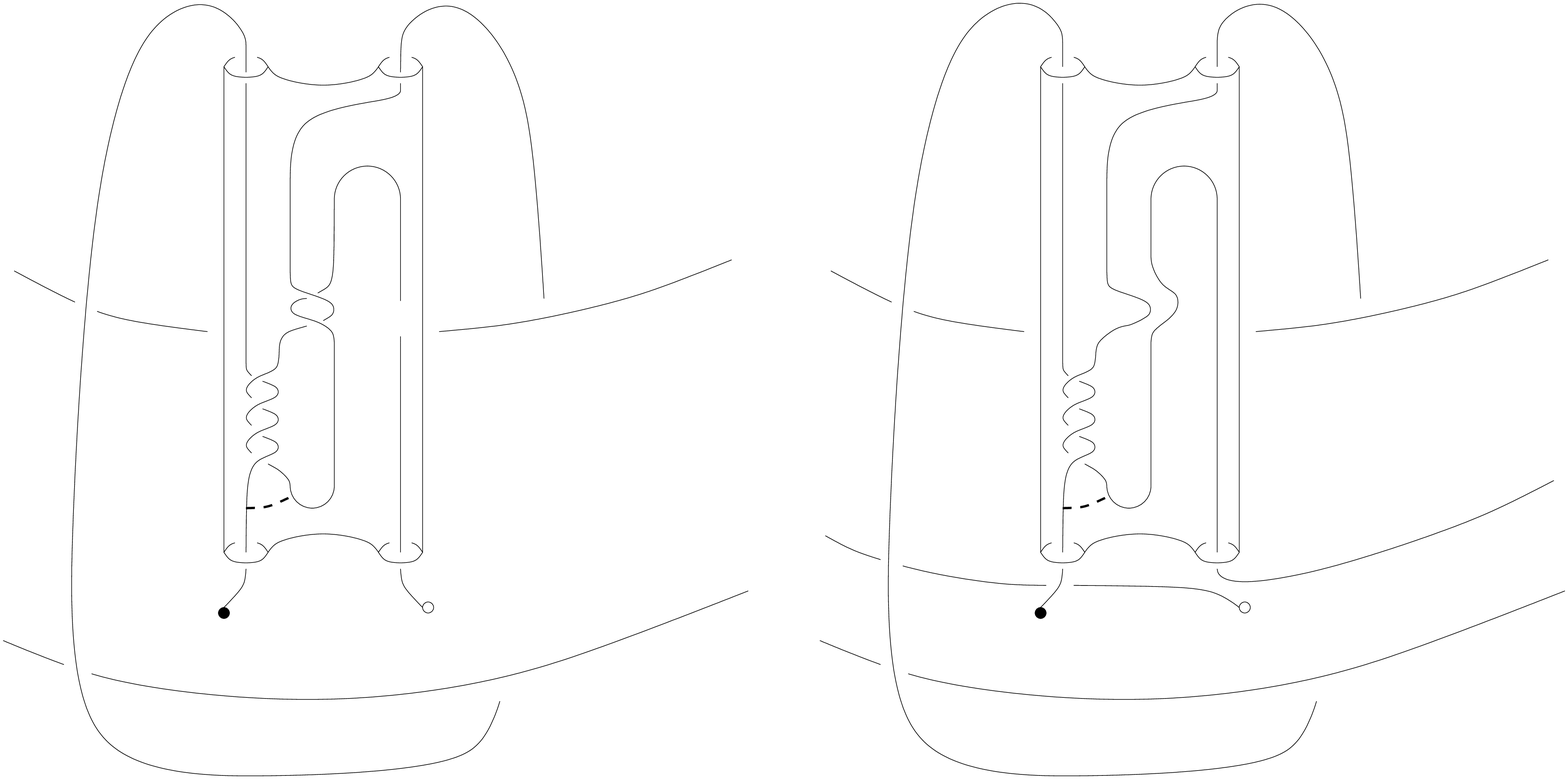}
\caption{Unwinding full twists of the middle two strands.}
\label{fig:unwinding}
\end{center}
\end{figure}
Figure~\ref{fig:unwinding}(a) shows a knot and tunnel produced by a cabling
construction, starting from a tunnel in the standard position seen in
Figure~\ref{fig:last_step}(c). Notice that the original tunnel arc seen in
Figure~\ref{fig:last_step}(c) now appears as an arc of the new knot--- this
is the ``swap'' part of the cabling construction. The arc labeled
$\alpha_W$ in Figure~\ref{fig:last_step}(c), that was an arc of the
original knot, is replaced by a standard tangle in the position shown in
Figure~\ref{fig:unwinding}(a). The new tunnel arc is the dotted arc at the
lower left of the two-bridge configuration. Provided that the original
tunnel arc was the canonical tunnel arc of the original knot, the new
tunnel arc is the canonical tunnel arc of the new knot. This is because the
cablings in the unique sequence producing a tunnel produce the canonical
tunnel arcs for each tunnel (see Section~\ref{sec:principal}).

The knot resulting from this cabling construction depends on the element
$\omega \in \mathcal{B}$, not just on its double coset $\langle
\delta_\ell,\sigma\rangle \omega \langle \delta_m,\sigma\rangle$.

Since the slope of any cabling producing a knot (rather than a
two-component link) is of the form $q/p$ with $q$ odd, Lemma~14.2
of~\cite{CMtree} shows that $q/p$ has a continued fraction expansion of the
form $[2a_1,b_1,2a_2,b_1,\ldots,\allowbreak 2a_n,b_n]$. So we can and will
assume that the standard tangle in Figure~\ref{fig:unwinding} has type of
the form $(2a_1,b_1,2a_2,\ldots,2a_n,b_n)$, and conseqeuntly the middle two
strands in the tangle have only full twists.

The first maneuver unwinds one left-hand full twist of the middle two
strands at the top of the braid, adding a letter $\delta_\ell^{-1}$ at the
beginning of the braid description of the previous knot. During the isotopy, the
knot cuts once across a core circle of $W$. The resulting knot is shown in
Figure~\ref{fig:unwinding}(b). If the twist is right-hand, the isotopy is
similar but a letter $\delta_\ell$ is added.

\begin{figure}
\labellist
\pinlabel {\large(a)} [B] at 260 -15
\pinlabel {\large(b)} [B] at 850 -15
\endlabellist
\begin{center}
\includegraphics[width=60ex]{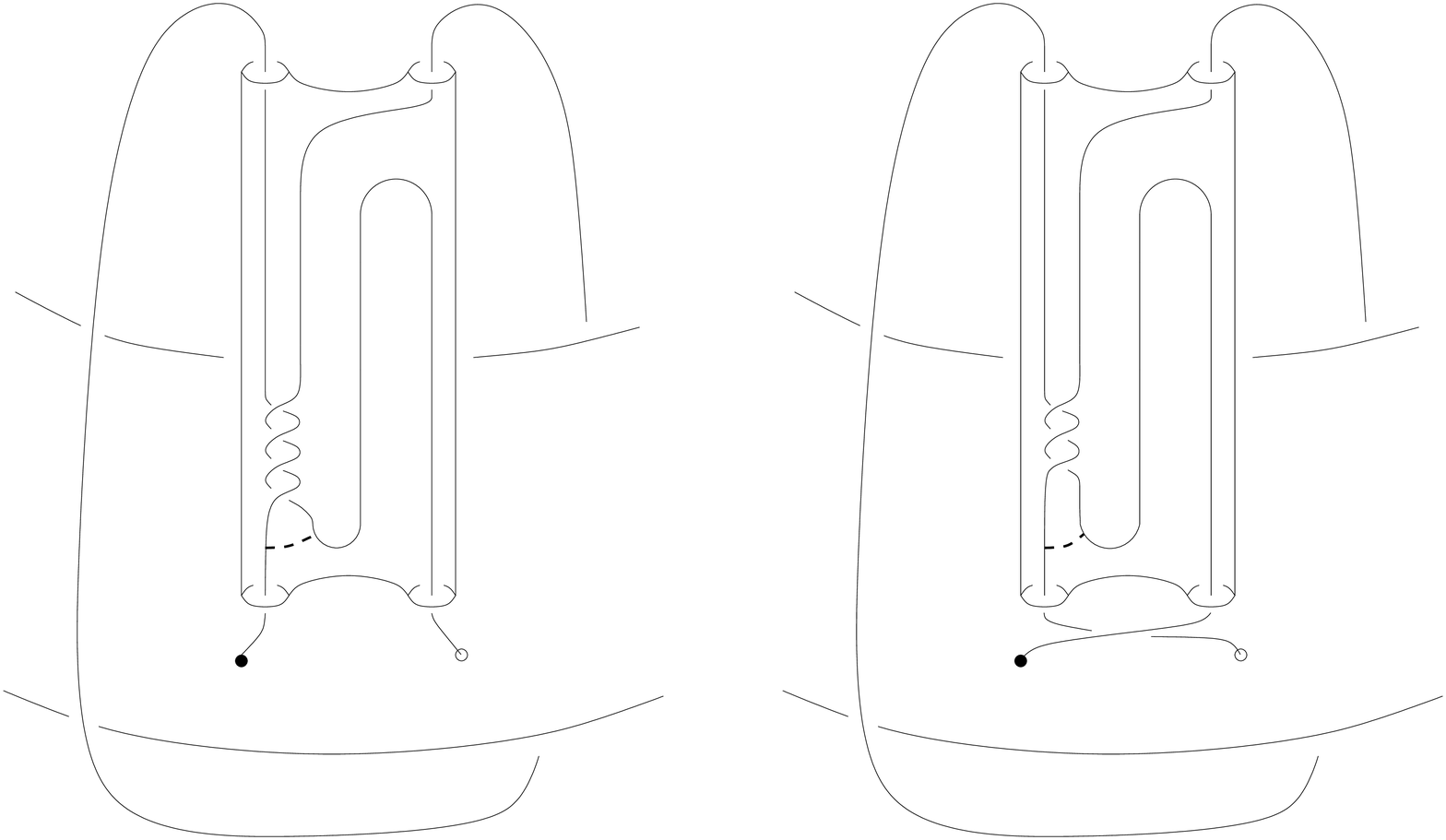}
\caption{Unwinding half twists of the left two strands.}
\label{fig:unwinding_sigma}
\end{center}
\end{figure}
The second maneuver is possible when there are no twists of the middle two
strands at the top of the braid, as in Figures~\ref{fig:unwinding}(b)
and~\ref{fig:unwinding_sigma}(a). It is similar to the first maneuver, but
unwinds half twists of the left two strands at the expense of an initial
powers of $\sigma$ to $\omega$.  During the isotopy, the knot need not pass
through a core circle of $W$; it can be fixed outside a small neighborhood
of the ball $B$ that contains the standard tangle. As seen in
Figure~\ref{fig:unwinding_sigma}(b), unwinding a single half twist adds an
initial letter $\sigma$ or $\sigma^{-1}$ to the braid description,
according as the half-twist is right-handed or left-handed.

If there are additional twists of the middle or left strands lying below
those shown in Figure~\ref{fig:unwinding}, they can be unwound by repeating
the previous two maneuvers. Thus the sequence of full twists of
the middle two strands and half twists of the left two strands unwinds to add
$\sigma^{b_n}\delta_\ell^{-a_n}\cdots \sigma^{b_1}\delta_\ell^{-a_1}$ at
the start of the braid description. The knot and tunnel are then in the position
in Figure~\ref{fig:last_step}(a), and the maneuver of
Section~\ref{sec:last_step} puts the knot and tunnel into the standard
position of Figure~\ref{fig:last_step}(c), adding $\delta_m\sigma$ to the
front of the braid description. This establishes our first main result:
\begin{theorem}[Unwinding Theorem] Suppose that a $(1,1)$-knot $K$ and its upper
$(1,1)$-tunnel are in standard position with braid description $w\in
  \mathcal{B}$. Perform a cabling construction that introduces a standard
  tangle of type $(2a_1,b_1,\ldots, 2a_n,b_n)$, positioned as shown in
  Figure~\ref{fig:unwinding}. Then using $(1,1)$-isotopy, the new knot and
  tunnel can be put into standard position with braid description
\[\delta_m\sigma\cdot \sigma^{b_n}\delta_\ell^{-a_n}\sigma^{b_{n-1}}\cdots
\sigma^{b_1}\delta_\ell^{-a_1}\cdot w\ .\]
\label{thm:unwinding}
\end{theorem}

\section{The Slope Theorem}
\label{sec:rho0}

To calculate the slope invariant of a cabling as in
Figure~\ref{fig:unwinding}, we must first find the slope-zero perpendicular
disk $\rho^0$ of $\rho$, where $\rho$ is the slope disk in
Figure~\ref{fig:rho0link} (a).  Then we determine the $\{\rho,
\rho^0\}$-slope of the standard tangle of type
$(2a_1,b_1,\ldots,2a_n,b_n)$.  In this section, we will carry these out,
leading to our second main result, Theorem~\ref{thm:winding}. It gives a
simple expression for the slope of a cabling construction of the type
considered in Theorem~\ref{thm:unwinding}. The expression involves an
integer that counts the number of turns the knot makes around the solid
torus $T\times I\cup V$. Definition~\ref{def:algebraic_winding_number} and
Proposition~\ref{prop:algebraic_winding_number} will show how to compute
this integer from a braid description of the $(1,1)$-position.

\begin{figure}
\labellist
\pinlabel $\rho$ [B] at 261 310
\pinlabel {$t$ times} [B] at 470 160
\pinlabel {\large(a)} [B] at 264 0
\pinlabel {\large(b)} [B] at 841 0
\endlabellist
\begin{center}
\includegraphics[width=67ex]{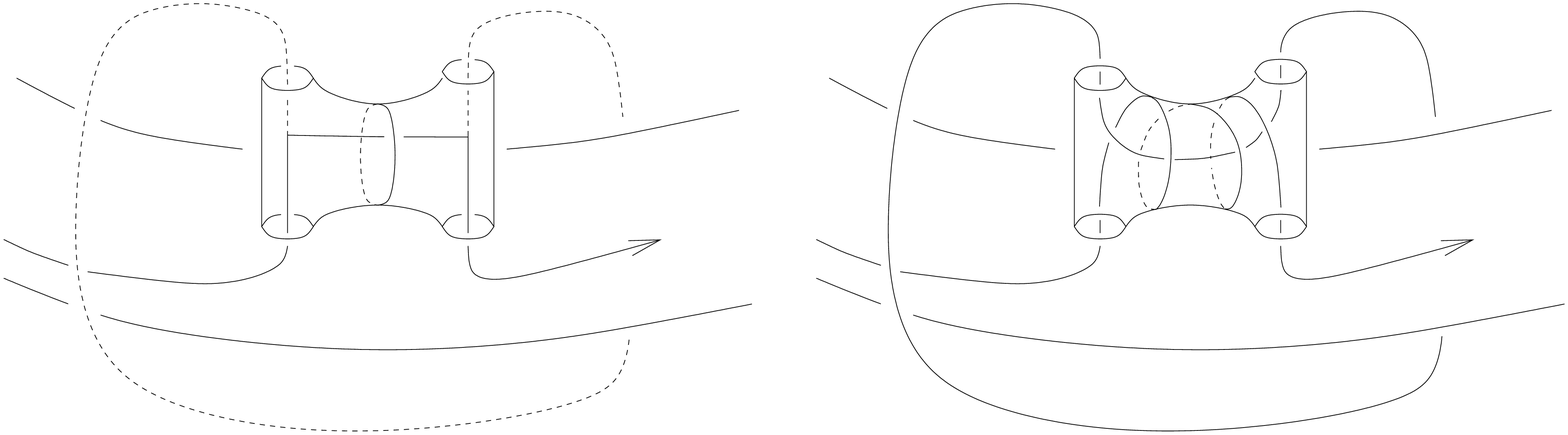}
\caption{Finding the slope-zero perpendicular disk $\rho^0$. The first drawing
shows $\rho$ and indicates the direction of algebraically positive
winding of the knot around $V$. The second shows the $0$-linking pair
for $\rho^0$, for the case $t=3$. The corresponding disk $\rho^0$ appears
in Figure~\ref{fig:rho0}(b).}
\label{fig:rho0link}
\end{center}
\end{figure}

Assuming that the upper tunnel of $K$ is in standard position as in
Figure~\ref{fig:last_step}(c), we always choose the orientation on $K$
to be directed over the top arch $\alpha_W$ from the black point to the
white point. The \textit{algebraic winding number} for this position of $K$
is defined to be the net number of turns that $K$ makes in
the direction of positive
orientation on the longitude $\ell$ of $T\times I\cup V$
(the direction indicated in Figure~\ref{fig:rho0link}(a)), that is, the
algebraic intersection of $K$ with a meridian disk of $T\times I\cup V$
bounded by the loop~$m$.
\begin{definition}\label{def:algebraic_winding_number}
For $\omega\in \mathcal{B}$, define $t(\omega)\in \Z$ as follows. For each
appearance of $\delta_\ell^{\epsilon}$ ($\epsilon=\pm1$) in $\omega$, write
$\omega=\omega_1\delta_\ell^{\epsilon}\omega_2$, and let $k$ be the total
exponent of $\sigma$ in $\omega_1$. Assign the value $(-1)^{k+1}\epsilon$
to this appearance of $\delta_\ell^\epsilon$, and sum these over all
appearances to give $t(\omega)$.
\end{definition}

\begin{proposition}\label{prop:algebraic_winding_number}
If $K=K(\omega)$, then $t(\omega)$ equals the algebraic winding number of
$K$.
\end{proposition}

\begin{proof}
For our designated orientation on $K$ and choice of direction of positive
winding, an initial letter $\delta_\ell$ in $\omega$ as in the example of
Figure~\ref{fig:braids}(a) would contribute $-1$ to the algebraic winding
number of $K$. If it were $\delta_\ell^{-1}$ it would contribute $+1$.
When $\delta_\ell^\epsilon$ is not the initial letter, each of the
appearances of $\sigma$ preceding an appearance of $\delta_\ell^\epsilon$
in $\omega$ reverses the direction in which the orientation of $K$ is
directed around the turn corresponding to this $\delta_\ell^\epsilon$
term. So if there are $k$ such appearances of $\sigma$, this appearance of
$\delta_\ell^\epsilon$ contributes $(-1)^k(-\epsilon)=(-1)^{k+1}\epsilon$
to the algebraic winding number. Apart from this effect of $\sigma$ on the
signs of these terms, the appearances of $\delta_m$ and $\sigma$ in
$\omega$ make no contribution to the algebraic winding number.
\end{proof}

We can now state our second main result.
\begin{theorem}[Slope Theorem]\label{thm:winding}
Let $K=K(\omega)$ be a knot in braid position with upper tunnel in standard
position as shown in Figure~\ref{fig:last_step}(c). Suppose that a cabling
construction as in Figure~\ref{fig:unwinding} is performed using a standard
tangle of type $(2a_1,b_1,\ldots,2a_n,b_n)$. Then the slope of the cabling
is given by the continued fraction $[2t(\omega)+2a_1,b_1,2a_2,\ldots,
  b_n]$.
\end{theorem}

\begin{proof}
Figure~\ref{fig:rho0link}(b) shows the link formed by the core circles of
the solid tori into which the slope disk in Figure~\ref{fig:rho0}(b) will
cut a handlebody neighborhood of the union of the knot and the tunnel. The
lower component is $(1,1)$-isotopic to $K$, while the upper component is a
core circle of the solid torus $W$. Let $t$ be the number of full
left-handed twists of the right half of $B$ needed to change the disk
$\rho^\perp$ in Figure~\ref{fig:rho0}(a) to the disk in
Figure~\ref{fig:rho0}(b). Recalling that the algebraic winding number of
$K$ is its algebraic intersection number with a meridian disk of $T\times
I\cup V$, we see that the linking number of the lower component with the
upper component is $t$ less than the algebraic winding number of $K$. If
we choose $t$ to equal this algebraic winding number, then the linking
number will be $0$, and therefore the disk in Figure~\ref{fig:rho0}(b) will
be the canonical zero-slope disk $\rho^0$. According to
Proposition~\ref{prop:algebraic_winding_number}, the algebraic winding
number of $K$ is $t(\omega)$, so the condition is that $t=t(\omega)$.
\begin{figure}
\labellist
\pinlabel $\tau_+$ [B] at -11 136
\pinlabel $\tau_-$ [B] at 180 136
\pinlabel $\tau_+$ [B] at 245 136
\pinlabel $\tau_-$ [B] at 435 136
\pinlabel $\rho$ [B] at 82 126
\pinlabel $\rho^\bot$ [B] at 180 80
\pinlabel $\rho^0$ [B] at 435 70
\pinlabel $\lambda_+$ [B] at -11 30
\pinlabel $\lambda_-$ [B] at 180 30
\pinlabel $\lambda_+$ [B] at 245 30
\pinlabel $\lambda_-$ [B] at 435 30
\pinlabel {\large(a)} [B] at 85 0
\pinlabel {\large(b)} [B] at 340 0
\endlabellist
\begin{center}
\includegraphics[width=58ex]{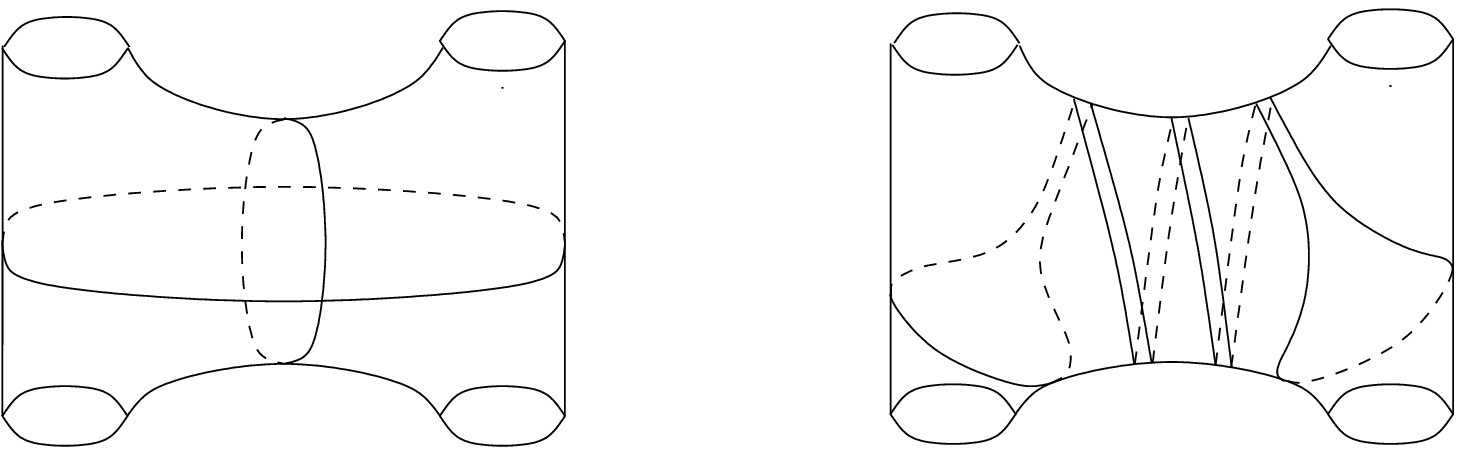}
\caption{A perpendicular disk $\rho^\perp$ and the
slope-zero perpendicular disk $\rho^0$ for the case $t=3$.}
\label{fig:rho0}
\end{center}
\end{figure}

Now consider a standard tangle $K_0$ of type $(2a_1,b_1,\ldots, 2a_n,b_n)$,
as shown in Figure~\ref{fig:standard_tangle}. Regard it as contained in
the portion of the handlebody shown in Figure~\ref{fig:rho0}, as in
Figure~\ref{fig:unwinding}.
Proposition~\ref{prop:slope_of_standard_tangle} gives the slope of
$K_0$ with respect to the pair $\{\rho,\rho^\perp\}$ in
Figure~\ref{fig:rho0}(a) to be $[2a_1,b_1,\ldots, 2a_n,b_n]$. We denote this
slope by $m(K_0,\{\rho,\rho^\perp\})$, and by $m(K_0,\{\rho,\rho^0\})$ the
slope with respect to $\{\rho,\rho^0\}$.

Let $u$ denote a full left-hand twist of the right-hand side of the ball
in Figure~\ref{fig:rho0}. We have already seen that
$u^t(\rho^\perp)=\rho^0$, and we note also that $u^t(\rho)=\rho$. In
the view of Figure~\ref{fig:standard_tangle}, $u$ is a full left-hand
twist of the bottom half of $B$, so $u^{-t}$ moves $K_0$ to the standard
tangle of type $(2t+2a_1,b_1,2a_2,b_2,\ldots, 2a_n,b_n)$. We can now
compute the slope of the cabling as
\begin{gather*}
m(K_0,\{\rho,\rho^0\}) = m(K_0,\{\rho,u^t(\rho^\perp)\})
= m(K_0,\{u^t(\rho),u^t(\rho^\perp)\})\\
= m(u^{-t}(K_0),\{\rho,\rho^\perp\}) = [2t+2a_1,b_1,\ldots,2a_n,b_n]
\end{gather*}
where Proposition~\ref{prop:slope_of_standard_tangle} gives the final equality.
\end{proof}

\begin{example}\label{ex:example}
Figure~\ref{fig:figure_10} shows the knot of Figure~9 of~\cite{CMtree}
moved into $(1,1)$-position. The upper right-hand drawing is the original
knot and its upper and lower tunnels $\tau_1$ and $\tau_2$. In the
$(1,1)$-position in Figure~9 of~\cite{CMtree}, $\tau_2$ is the upper
tunnel. From the bottom-right drawing, we read off a braid description of
$\tau_2$ as
$\omega=\delta_m^{-1}\sigma^{-1}\delta_\ell\delta_m^{-1}\sigma^3
\delta_\ell^{-1}=\sigma\delta_m\delta_\ell\delta_m^{-1}\sigma^3\delta_\ell^{-1}\sim
\delta_m\delta_\ell\delta_m^{-1}\sigma^3\delta_\ell^{-1}$.
\begin{figure}
\labellist
\pinlabel $\tau_1$ [B] at 66 128
\pinlabel $\tau_2$ [B] at 66 188
\pinlabel $\tau_1$ [B] at 194 153
\pinlabel $\tau_2$ [B] at 223 176
\pinlabel $\tau_1$ [B] at 60 17
\pinlabel $\tau_2$ [B] at 61 70
\pinlabel $\tau_1$ [B] at 200 12
\pinlabel $\tau_2$ [B] at 199 75
\endlabellist
\begin{center}
\includegraphics[width=65 ex]{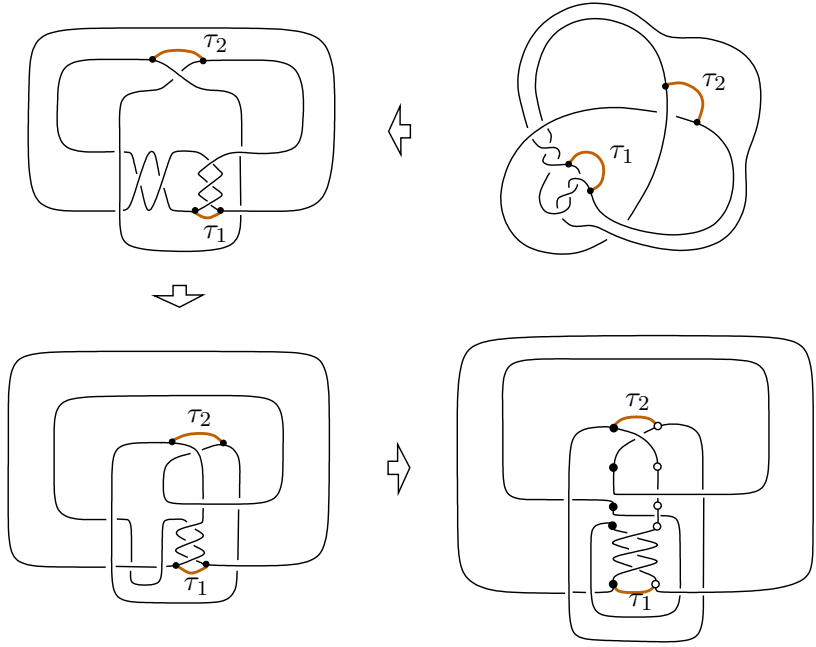}
\caption{Putting a knot into $(1, 1)$-position}
\label{fig:figure_10}
\end{center}
\end{figure}

To compute the slope invariants for
the tunnel $\tau_2$, we use the relation
$\delta_m^{-1}=\sigma\delta_m\sigma$ to put $\omega$ into the form
\[\delta_m\sigma\cdot \omega_1(\delta_\ell,\sigma)\cdot \delta_m\sigma\cdot
\omega_0(\delta_\ell,\sigma)=\delta_m\sigma\cdot
\sigma^{-1}\delta_\ell\sigma \cdot \delta_m\sigma\cdot
\sigma^3\delta_\ell^{-1}\ .\] We now use Theorems~\ref{thm:unwinding}
and~\ref{thm:winding} to read off the slopes. The first cabling starts from
the trivial knot $K$, which has algebraic winding number $t(1)=0$. Since
the cabling corresponds to the portion $\delta_m\sigma\cdot
\omega_0(\delta_\ell,\sigma)=\delta_m\sigma\cdot \sigma^3\delta_\ell^{-1}$,
Theorem~\ref{thm:unwinding} shows that the standard tangle used in the
cabling is of type $(2a_1,b_1)=(2,3)$.  By Theorem~\ref{thm:winding}, the
ordinary slope of the first cabling is given by the continued fraction
$[0+2,3]=7/3$, so the simple slope is $[3/7]$ in $\Q/\Z$.  The second cabling begins
with this knot, so has algebraic winding number $t(\delta_m\sigma\cdot
\sigma^3\delta_\ell^{-1})= (-1)^{4+1}\cdot (-1)=1$.  From
Theorem~\ref{thm:unwinding}, we have $a_1=0$, $b_1=1$, $a_2=-1$, and
$b_2=-1$, so by Theorem~\ref{thm:winding} the second cabling slope is given
by the continued fraction $[2+0,1,-2,-1]=7/2$. Therefore the slope sequence
of $\tau_2$ is $[3/7], 7/2$.

The tunnel $\tau_1$ is the upper tunnel of the $(1,1)$-position
described by the reverse braid of~$\omega$,
\[ \delta_m^{-1} \sigma^3\delta_\ell^{-1}\delta_m\delta_\ell\sim
\delta_m\sigma\cdot \sigma^3\delta_\ell^{-1}\cdot \delta_m\sigma\cdot
\sigma^{-1}\delta_\ell\ ,\] giving the first slope of $\tau_1$ to be
$[0+(-2),-1]=-3$ and consequently its simple slope to be $[-1/3]=[2/3]\in \Q/\Z$.
For the second slope, we have
$t(\delta_m\delta_\ell)=(-1)^{0+1}\cdot 1=-1$, $a_1=1$, and $b_1=3$, giving
the slope $[2\cdot(-1)+2,3]=1/3$. Therefore the cabling slope
sequence of the lower tunnel is $[2/3],1/3$.
\end{example}

\section{Tunnels of $2$-bridge knots}
\label{sec:2bridge}

In this section we will give braid descriptions of the tunnels of
$2$-bridge knots, and use them to calculate the slope invariants. We
obtain, of course, the same values as in the calculation of \cite[Section
15]{CMtree}. In Theorem~\ref{thm:slope_sequence_characterization} we give
a characterization of which sequences of rational numbers (with initial
term in $\Q/\Z$) occur as slope sequences of tunnels of $2$-bridge knots.

A convenient reference for the tunnels of $2$-bridge knots is K. Morimoto
and M. Sakuma~\cite[Section (1.6)]{M-S}. In \cite[Section (1.1)]{M-S}, the
authors give a definition of \textit{dual} tunnels, and the discussion in
\cite[Section (1.2)]{M-S} shows that two tunnels of a knot in $S^3$ are
dual exactly when they are the upper and lower tunnels for the same
$(1,1)$-position of the knot. As we saw in Section~\ref{sec:braid_group},
the dual of a tunnel given by a braid description $\omega$ has a braid
description by the reverse braid of $\omega$.

\begin{figure}
\labellist
\pinlabel $U$ [B] at -10 45
\pinlabel $US$ [B] at 66 83
\pinlabel $L$ [B] at 376 50
\pinlabel $LS$ [B] at 354 78
\pinlabel $2a_1$ [B] at 100 45
\pinlabel $2a_2$ [B] at 198 45
\pinlabel $b_1$ [B] at 150 12
\pinlabel $b_n$ [B] at 335 12
\endlabellist
\begin{center}
\includegraphics[width=58 ex]{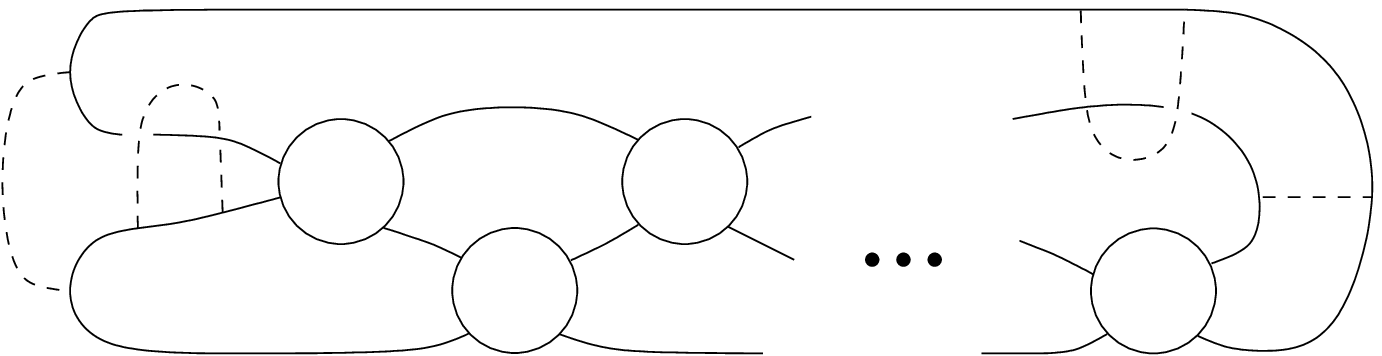}
\caption{A $2$-bridge knot with classifying invariant given by the
continued fraction
$1/[2a_1,b_1,2a_2,\ldots,2a_n,b_n]$. The upper and lower tunnels
$U$ and $L$ and the upper and lower semisimple tunnels $US$ and $LS$ are
shown.}
\label{fig:2bridge}
\end{center}
\end{figure}
Figure~\ref{fig:2bridge} shows a $2$-bridge knot, where the regions labeled
$2a_i$ indicate $2a_i$ left-hand half-twists and those labeled $b_i$
indicate $b_i$ right-hand half-twists. The upper, lower, upper semisimple,
and lower semisimple tunnels are shown; from \cite[Section (1.6)]{M-S}, the
upper semisimple and lower simple tunnels are dual, as are the lower
semisimple and upper simple tunnels.

This position is assigned to the rational number $a/b$ given by the continued
fraction $[2a_1,b_1,2a_2,\ldots,2a_n,b_n]$. Changing the position by
$(1,1)$-isotopy if need be, we may assume that $2a_1$ and $b_n$ are
nonzero (indeed we may assume that no $a_i$ or $b_i$ is zero, although for
some calculations it is convenient to allow zero values), and we always
choose $a$ positive, so have $0<|b|<a$. Also, $a$ is odd (the values when
$a$ is even correspond to $2$-bridge links).

Notice that there is an isotopy moving the knot in Figure~\ref{fig:2bridge}
to the position given by the continued fraction
$[-b_n,-2a_n,-b_{n-1},\ldots,-b_1,-2a_1]$. The first step is to move the
top horizontal strand down to the bottom. The twists $b_i$
then appear in the middle and the $2a_i$ at the top. Then the
entire knot is rotated until it looks as in Figure~\ref{fig:2bridge} except
with the twists $b_i$ in the middle and the $2a_i$
at the bottom; the minus signs are due to the convention about which
directions of twists are considered to be positive for the middle versus
the bottom two strands. The upper tunnel for the second position is the
lower tunnel for the original position. Similarly, the upper semisimple
tunnel for the first position is the lower semisimple tunnel for the
second.

If $[2a_1,b_1,\ldots, 2a_n,b_n]=a/b$, then
$[-b_n,-2a_n,-b_{n-1},\ldots,-b_1,-2a_1]$ is $a/b'$ where $0<|b|<a$,
$0<|b'|<a$, and $bb'\equiv 1\bmod a$. This and many other basic facts
about continued fraction expansions can be verified using \cite[Lemma
14.3]{CMtree}. For suppose we use Lemma~14.2 of \cite{CMtree} (which is
itself a consequence of Lemma~14.3 of \cite{CMtree}) to write
$[2a_1,b_1,\ldots,2a_n,b_n]=a/b$. By Lemma~14.3 of \cite{CMtree}, we have
$U^{2a_1}L^{b_1}\cdots U^{2a_n}L^{b_n}=\begin{bmatrix} a & r \\ b &
s\end{bmatrix}$ where
$[-b_n,-2a_n,-b_{n-1},\ldots,-b_1,-2a_1]
=-[b_n,2a_n,b_{n-1},\ldots,b_1,2a_1]=a/(-r)$. Since $as-rb=1$, we have
$(-r)b=1\bmod a$.

The $2$-bridge knot is actually classified up to isotopy by the pair of
(possibly equal) values $b/a$ and $b'/a$ in $\Q/\Z$. Replacing $a/b$ by
$a/(b\pm a)$, if necessary, and applying Lemma~14.2 of~\cite{CMtree}, we
may assume that all terms in the continued fraction expansion of $a/b$ are
even. The corresponding $2$-bridge position, having only full twists of the
left two strands and the middle two strands, is called the \textit{Conway
  position} of the $2$-bridge knot.

\begin{proposition}\label{prop:2bridge_slopes}
The lower simple tunnel has slope invariant
\[m_0=[1/[2a_1,b_1,\ldots, 2a_n,b_n]]\ ,\]
and the upper simple tunnel has
slope invariant
\[m_0=[1/[-b_n,-2a_n,-b_{n-1},\ldots,-b_1,-2a_1]]\ .\]\par
\end{proposition}

\begin{proof}
Figure~\ref{fig:lower_tunnel_slope}, a case of Figure~\ref{fig:unwinding}(a),
shows a $2$-bridge knot $K$ obtained from the trivial knot $K_0$ by a
single cabling. The tunnel arc is the lower simple tunnel of $K$. Since the
algebraic winding number $t(K_0)$ is $0$, Theorem~\ref{thm:winding} gives
the slope of this cabling to be $[2a_1,b_1,\ldots, 2a_n,b_n]$, so the
simple slope of the lower tunnel is $[1/[2a_1,b_1,\ldots,
2a_n,b_n]]$. Since the upper simple tunnel is the lower simple tunnel
for the position of $K$ corresponding to the continued fraction
$[-b_n,-2a_n,-b_{n-1},\ldots,-b_1,-2a_1]$, the simple slope of the upper
tunnel is as given in the proposition. (To apply Theorem~\ref{thm:winding},
the position would have to be moved by isotopy to change the $-b_i$ to be
even, but this would not change the value of the continued fraction.)
\begin{figure}
\labellist
\pinlabel $2a_1$ [B] at 300 370
\pinlabel $b_1$ [B] at 247 280
\endlabellist
\begin{center}
\includegraphics[width=40ex]{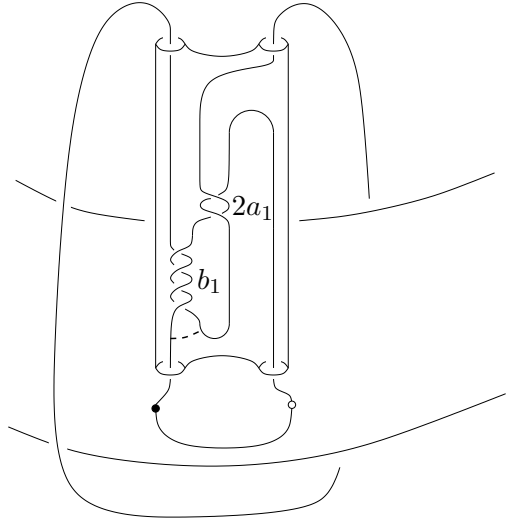}
\caption{The slope calculation for the lower tunnel of a $2$-bridge knot.}
\label{fig:lower_tunnel_slope}
\end{center}
\end{figure}
\end{proof}

From Proposition~\ref{prop:2bridge_slopes}, we have
\begin{corollary} Let the rational invariant of the $2$-bridge knot be
given by the continued fraction $a/b=[2a_1,b_1,2a_2,\ldots, 2a_n,b_n]$,
with $0<b<a$. Let $b'$ be the integer with $0<b'<a$ and $bb'\equiv 1\mod
a$. Then the simple slope of the upper tunnel of $K$ is $[b'/a]$, and the
simple slope of the lower tunnel is $[b/a]$.
\label{coro:simple_slopes}
\end{corollary}

We turn now to the semisimple tunnels, whose slope invariants were
calculated in~\cite{CMtree}. We will obtain a braid description such that
the upper semisimple tunnel is the upper tunnel for the associated
$(1,1)$-position, then use it to recover the slope calculation of
\cite{CMtree}. We will also prove a new result,
Theorem~\ref{thm:slope_sequence_characterization}, which characterizes the
slope sequences of these tunnels.

Braid descriptions of these $(1,1)$-positions were given by A. Cattabriga
and M. Mulazzani~\cite{Cat-Mul} and more recently in the dissertation of
A. Seo~\cite{Seo}. Here is the braid description that we will use:
\begin{lemma} The braid word
$\delta_m^{-a_1}\sigma^{b_1}\cdots\delta_m^{-a_n}\sigma^{b_n}\delta_\ell^{-1}$
  describes a $(1,1)$-position of the
$2$-bridge knot $K$ given by
  $[2a_1,b_1,\ldots,2a_n,b_n]$. The upper tunnel of this
$(1,1)$-position is the upper semisimple tunnwl of $K$.\par
\label{lem:semisimple_braid_word}
\end{lemma}
\longpage

A quick way to obtain this braid description is to use the fact that the
upper semisimple tunnel is dual to the lower simple tunnel. As seen in
Figure~\ref{fig:lower_tunnel_slope}, the lower simple tunnel is obtained
from the upper tunnel of the trivial knot by a single cabling of type
$(2a_1,b_1,\ldots,2a_n,b_n)$. By Theorem~\ref{thm:unwinding}, this
$(1,1)$-position is described by the braid
$\delta_m\sigma^{b_n+1}\delta_\ell^{-a_n}\cdots
\sigma^{b_1}\delta_\ell^{-a_1}$. Since the upper semisimple tunnel is dual
to the lower simple tunnel, it is the upper tunnel of the $(1,1)$-position
described by the reverse of this word, which~is
\[\delta_m^{-a_1}\sigma^{b_1}\cdots
\delta_m^{-a_n}\sigma^{b_n+1}\delta_\ell
\sim \delta_m^{-a_1}\sigma^{b_1}\cdots
\delta_m^{-a_n}\sigma^{b_n}\delta_\ell^{-1}\ .\]

\begin{figure}
\labellist
\pinlabel $m$ [B] at 203 415
\pinlabel $\ell$ [B] at 180 350
\pinlabel $\delta^{-1}_m$ [B] at 315 130
\pinlabel $\sigma$ [B] at 137 35
\pinlabel $\ell$ [B] at 573 262
\pinlabel $m$ [B] at 655 195
\endlabellist
\begin{center}
\includegraphics[width=52ex]{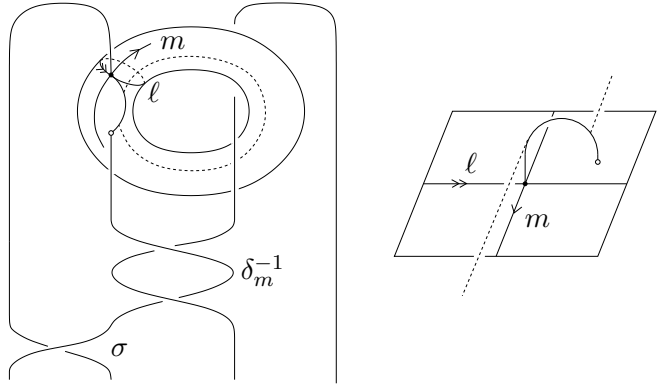}
\caption{Calculation of the braid description for the upper semisimple tunnel of a
$2$-bridge knot, as seen from the outside and from the inside of the
standard torus.}
\label{fig:2bridge_word}
\end{center}
\end{figure}
\begin{figure}
\labellist
\pinlabel $m$ [B] at 198 305
\pinlabel $\ell$ [B] at 180 241
\endlabellist
\begin{center}
\includegraphics[width=30ex]{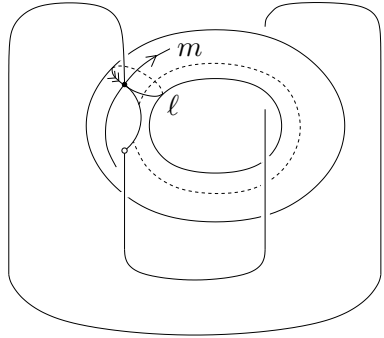}
\caption{The initial position of the trivial tunnel as an upper semisimple
tunnel has braid description $\delta_\ell^{-1}$.}
\label{fig:initial}
\end{center}
\end{figure}
A second, perhaps more satisfying way to obtain
Lemma~\ref{lem:semisimple_braid_word} is to see the braid directly.
Figure~\ref{fig:2bridge_word} shows the setup. The second drawing shows the
view from inside $W$, similar to the view of Figure~\ref{fig:braids}, and
the first shows the view from $T\times I$ looking at $W$ from the
outside. Observe that a full twist of the middle two strands represents
the braid that moves the white point around $m$ in the positive direction,
which is $\delta_m^{-1}$, and the half twist of the left two strands
represents $\sigma$.

Figure~\ref{fig:initial} shows the tunnel of the trivial knot as
the upper semisimple tunnel of the trivial $2$-bridge knot. It is the upper
tunnel for the $(1,1)$-position of the trivial knot described by the braid
$\delta_\ell^{-1}$ that moves the white point around $\ell$ in the positive
direction, that is, $K(\delta_\ell^{-1})$.

Now, modify the trivial knot in Figure~\ref{fig:initial} by inserting a
standard tangle of type $(2a_1,b_1,\ldots,2a_n,b_n)$ into $T\times I$, in
the position seen in Figure~\ref{fig:2bridge_word}. The portion labeled
$\delta_m^{-1}$ in Figure~\ref{fig:2bridge_word} will be $2a_1$ left-hand
half twists, the portion labeled $\sigma$ will be $b_1$ right-hand half
twists, below this will be $2a_2$ left-hand half twists, and so on, ending
with $b_n$ half twists and then the $\delta_\ell^{-1}$ already present.
This produces the knot seen in Figure~\ref{fig:2bridge}, and the upper
semisimple tunnel seen there, and the braid description of
Lemma~\ref{lem:semisimple_braid_word}.

We are now ready to calculate the slope invariants. By allowing the
possibility that $b_i=0$, we may assume that every $a_i$ is $\pm 1$ (since
continued fractions have the property that
$[\cdots,n_1+n_2,\cdots]=[\cdots,n_1,0,n_2,\cdots]$). We may further assume
that if the last term $b_n=\pm1$ then $a_n$ and $b_n$ have the same sign.

It is convenient to reindex the continued fraction as
$[2a_d,b_d,2a_{d-1},\ldots,\allowbreak 2a_0,b_0]$. We first consider four
cases with $i\geq 1$:
\smallskip

\noindent\textsl{Case I:} $a_i=1$, $a_{i-1}=1$
\smallskip

In this case the braid appears as
\[\omega_1\;\delta_m^{-1}\sigma^{b_i}\;\delta_m^{-1}\sigma^{b_{i-1}}\;\omega_2
= \omega_1\;\sigma \;(\delta_m\sigma)\sigma^{b_i+1}\;\delta_m\sigma^{b_{i-1}+1}\;\omega_2\]
and the cabling corresponding to $a_i$ has slope $[2t(\omega),b_i+1]$, where
$\omega=\delta_m\sigma^{b_{i-1}+1}\;\omega_2$.
\smallskip

\noindent\textsl{Case II:} $a_i=-1$, $a_{i-1}=1$
\smallskip

The braid is
\[\omega_1\; \delta_m\sigma^{b_i}\;\delta_m^{-1}\sigma^{b_{i-1}}\;\omega_2
= \omega_1\; (\delta_m\sigma)\sigma^{b_i}\;\delta_m\sigma^{b_{i-1}+1}\;\omega_2\]
and the cabling corresponding to $a_i$ has slope
$[2t(\omega),b_i]$, where again $\omega=\delta_m\sigma^{b_{i-1}+1}\;\omega_2$.
\smallskip

\noindent\textsl{Case III:} $a_i=1$, $a_{i-1}=-1$
\smallskip

The braid is
\[ \omega_1\; \delta_m^{-1}\sigma^{b_i}\;\delta_m\sigma^{b_{i-1}}\;\omega_2
= \omega_1\; \sigma \;(\delta_m\sigma)\sigma^{b_i}\;\delta_m\sigma^{b_{i-1}}\;\omega_2\]
and the cabling corresponding to $a_i$ has slope $[2t(\omega),b_i]$,
but this time $\omega=\delta_m\sigma^{b_{i-1}}\;\omega_2$.\par
\smallskip

\noindent\textsl{Case IV:} $a_i=-1$, $a_{i-1}=-1$
\smallskip

The braid is
\[\omega_1\; \delta_m\sigma^{b_i}\;\delta_m\sigma^{b_{i-1}}\;\omega_2
= \omega_1\; (\delta_m\sigma)\sigma^{b_i-1}\;\delta_m\sigma^{b_{i-1}}\;\omega_2\]
and the cabling corresponding to $a_i$ has slope
$[2t(\omega),b_i-1]$, where again $\omega=\delta_m\sigma^{b_{i-1}}\;\omega_2$.\par
\medskip

For the initial cabling, we have

\noindent\textsl{Case V:} $a_0=1$
\smallskip

The braid is
\[ \cdots \sigma^{b_1}\,\delta_m^{-1}\sigma^{b_0}\delta_\ell^{-1}
= \cdots \sigma^{b_1+1} \,(\delta_m\sigma)\sigma^{b_0}\;\delta_\ell^{-1}\]
and the initial cabling has simple slope $[1/[2,b_0]]=[b_0/(2b_0+1)]$
\smallskip

\noindent\textsl{Case VI:} $a_0=-1$
\smallskip

The braid is
\[\cdots \sigma^{b_1}\,\delta_m\sigma^{b_0}\delta_\ell^{-1}
= \cdots \sigma^{b_1}\,(\delta_m\sigma)\sigma^{b_0-1}\;\delta_\ell^{-1}\]
and the initial cabling has simple slope
$[1/[2,b_0-1]]=[(b_0-1)/(2b_0-1)]$.\par
\medskip

From Cases~V and~VI, we have $m_0=[b_0/(2b_0+1)]$ or
$m_0=[(b_0-1)/(2b_0-1)]$ according as $a_0$ is $1$ or $-1$.

To compute the remaining $m_i$, we assume that the knot is in Conway
position so that all the $b_i$ are even. We then have
\[t(\omega_2)
=t(\delta_m^{-a_{i-2}}\sigma^{b_{i-2}}\cdots
\delta_m^{-a_0}\sigma^{b_0}\delta_\ell^{-1})
=(-1)^{b_{i-2}+\cdots+b_0}=1\ ,\]
and from Cases~I-IV, using the fact that $b_{i-1}$ is even,
$t(\omega)$ equals $-1$ when $a_{i-1}=1$ and $1$ when $a_{i-1}=-1$. That is,
$t(\omega)=-a_{i-1}$. Summarizing, we have
\begin{proposition}
Let $K$ be in the $2$-bridge position corresponding to the continued
fraction $[2a_d,2b_d,\ldots,2a_0,2b_0]$, with $b_0\neq 0$ and each $a_i=\pm
1$. Then the slope invariants of the upper semisimple tunnel of $K$ are as
follows:
\begin{enumerate}
\item[(i)] $m_0=\left[\displaystyle\frac{2b_0}{4b_0+1}\right]$ or
$m_0=\left[\displaystyle\frac{2b_0-1}{4b_0-1}\right]$
according as $a_0$ is $1$ or $-1$.
\item[(ii)] For $1\leq i\leq d$, $m_i=-2a_{i-1}+1/k_i$, where
\begin{enumerate}
\item[(a)] $k_i=2b_i+1$ if $a_i=a_{i-1}=1$,
\item[(b)] $k_i=2b_i$ if $a_i$ and $a_{i-1}$ have opposite signs, and
\item[(c)] $k_i=2b_i-1$ if $a_i=a_{i-1}=-1$.
\end{enumerate}
\end{enumerate}
\label{prop:semisimple_slopes}
\end{proposition}
\noindent This agrees with the calculation obtained in \cite[Section 15]{CMtree}.

Using Proposition~\ref{prop:semisimple_slopes}, we can characterize the
slope sequences of semi\-simple tunnels of $2$-bridge knots.
\begin{theorem} Let $m_0,m_1,\ldots\,$, $m_d$ be a sequence with $m_0\in
\Q/\Z$ and $m_i\in \Q$ for $i>0$. Then $m_0,m_1,\ldots\,$, $m_d$ is the
slope sequence for a semisimple tunnel of a $2$-bridge knot if and only if
it satisfies the following:
\begin{enumerate}
\item[(i)] $m_0=\left[ \displaystyle\frac{n_0}{2n_0+1}\right]$ for some $n_0\notin
\{-1,0\}$.
\item[(ii)] For $i>0$, $m_i=\pm 2 + \displaystyle\frac{1}{k_i}$ for some
integer $k_i\neq 0$.
\item[(iii)] $m_1$ is positive or negative according as $n_0$ is odd
or even.
\item[(iv)] For $1\leq i\leq d$, $m_i$ has the same sign as $m_{i-1}$ if and
only if $k_{i-1}$ is odd.
\end{enumerate}
\label{thm:slope_sequence_characterization}
\end{theorem}
\begin{proof}
First assume that this is a slope sequence for a semisimple tunnel. Part
(i) follows from Proposition~\ref{prop:semisimple_slopes}(i), with the
excluded cases corresponding to the cases when $b_0=0$. Part (ii) is
immediate from Proposition~\ref{prop:semisimple_slopes}(ii). In
Proposition~\ref{prop:semisimple_slopes}(ii), $m_1$ has the opposite sign
from $a_0$, and Proposition~\ref{prop:semisimple_slopes}(i) shows that
$a_0$ is negative or positive according as $n_0$ is odd or even. This
establishes part (iii). For part (iv),
Proposition~\ref{prop:semisimple_slopes}(ii) shows that the signs of $m_i$
and $m_{i-1}$ differ exactly when $a_{i-1}$ and $a_{i-2}$ have opposite
signs. By Proposition~\ref{prop:semisimple_slopes}(ii)(a)(b)(c), this is
exactly the case when $k_{i-1}=2b_{i-1}$, and when $a_{i-1}$ and
$a_{i-2}$ have the same sign, $k_{i-1}=2b_{i-1}\pm 1$ which is odd.

For the converse, given a sequence $m_0,\ldots\,$, $m_d$ as in the
statement of the Proposition, we will construct the continued fraction
expansion $a/b=[2a_d,2b_d,\ldots,2a_0,2b_0]$ for the Conway position, with each
$a_i=\pm1$. Let $a_0$ be $-1$ or $1$ according to whether $n_0$ is odd or
even. For ascending $i$ with $1\leq i\leq d$, let $a_i$ be $a_{i-1}$ if
$k_i$ is odd, and $-a_{i-1}$ if it is even. For the $b_i$, put
$2b_0=n_0$ or $n_0+1$ according as $n_0$ is even or odd. Cases~I-IV now
determine the choices of the remaining $2b_i$ which produce the correct
values for the $m_i$: When $a_i=-a_{i-1}$, $k_i$ is even and Cases~II
and~III give $2b_i=k_i$. When $a_i=a_{i-1}$, $k_i$ is odd and Cases~I
and~IV show that $2b_i=k_i-1$ if $a_i=1$ and $2b_i=k_i+1$ if $a_i=-1$.
\end{proof}

\section{Semisimple tunnels of torus knots}
\label{sec:torus_knots}

To set notation, consider a (nontrivial) $(p,q)$-torus knot $K_{p,q}$,
contained in our standard torus $T$. It represents $p$ times a generator in
$\pi_1(V\cup T\times I)$, and $q$ times a generator in $\pi_1(W)$.

The tunnels of torus knots were classified by M. Boileau, M. Rost, and
H. Zieschang \cite{B-R-Z} and Y. Moriah~\cite{Moriah}.  The \textit{middle
  tunnel} of $K_{p,q}$ is represented by an arc in $T$ that meets $K_{p,q}$
only in its endpoints. The \textit{upper tunnel} of $K_{p,q}$ is
represented by an arc $\alpha$ properly imbedded in $W$, such that the
circle which is the union of $\alpha$ with one of the two arcs of $K_{p,q}$
with endpoints equal to the endpoints of $\alpha$ is a deformation retract
of $W$. The \textit{lower tunnel} is like the upper tunnel, but
interchanging the roles of $V\cup T\times I$ and~$W$. For some choices of
$p$ and $q$, some of these tunnels are equivalent.

A braid description for torus knots was obtained by A. Cattabriga and
M. Mulazzani~\cite[Section 4]{Cat-Mul1}. Here we will use a similar
description due to A. Seo~\cite{Seo}. Fix $(p,q)$ relatively prime, and
suppose for now that $p,q\geq 2$. In $\R^2$ we construct a polygonal path
$P_{p,q}$ from $(0,0)$ to $(p,q)$, as indicated in
Figure~\ref{fig:torus_word} for the cases when $(p,q)$ is $(3,7)$ and
$(7,3)$, as follows: Regard $\R^2$ as made up of squares with side length
$1$, whose corners have integer coordinates.  Consider the rectangle $R$
with corners $(0,0)$ and $(p,q)$, and let $S$ be the union of the squares
in $\R^2$ whose bottom sides contain no points above the line containing
$(0,0)$ and $(p,q)$. Then $P_{p,q}$ is $R\cap \partial S$.

\begin{figure}
\begin{center}
\includegraphics[width=53ex]{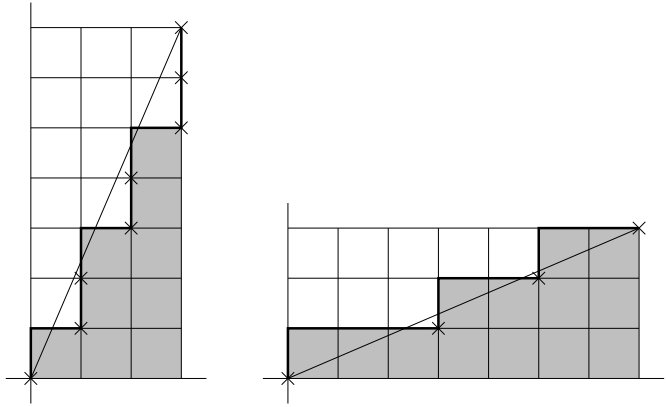}
\caption{The points $(p_k,k)$ for $0\leq k\leq q$ for
the cases $(p,q)=(3,7)$ and $(p,q)=(7,3)$.}
\label{fig:torus_word}
\end{center}
\end{figure}
Explicitly, for $0\leq k\leq q$ put $p_k=\lceil kp/q\rceil$.  The points
$(p_k,k)$ are indicated in Figure~\ref{fig:torus_word}. The intersection
point of the diagonal of $R$ with the line $y=k$ has $x$-coordinate $kp/q$,
so $\lceil kp/q\rceil$ is the $x$-coordinate of the first integral lattice
point on $y=k$ that lies on or to the right of this intersection point.
The path $P_{p,q}$ is the union for $1\leq i\leq q$ of the segments (some
of which may have length $0$) from $(p_{i-1},i-1)$ to $(p_{i-1},i)$ and
from $(p_{i-1},i)$ to $(p_i,i)$.

We will now use $P_{p,q}$ to obtain a braid description of $K_{p,q}$.  As
usual, the braid portion will start in $T$.  Referring again to
Figure~\ref{fig:torus_word}, we may assume that the plane picture is drawn
so that the lifts of the fixed arc $\alpha$ in $T$ from the black point to
the white point are short straight line segments positioned so that each
meets only the translate of the diagonal of $R$ that contains its lower
left point (that is, its black point), and apart from that point lies below
that translate.

Consider the braid $\beta$ with the following description. As we descend in
$T\times I$, the $T$-coordinate of the white point stays fixed, while the
$T$-coordinate of the black point moves \textit{backward} along the
$(p,q)$-curve which is the image of the diagonal of the rectangle (that is,
the lift to $\R^2$ of its path in $T$ starting $(p,q)$ travels along the
diagonal of $R$ to $(0,0)$). The choice of the backward direction is not
essential, but leads to a simpler calculation. Since each lift of $\alpha$
lie below the translate of the diagonal that it meets, the diagonal of $R$
is isotopic, not crossing any lift of $\alpha$ and in particular not
crossing any lift of the white point, to $P_{p,q}$. This implies that
$\beta$ is represented by the braid word whose letters correspond to the
horizontal and vertical steps along $P_{p,q}$ (starting from the upper
right), with each downward step being $\delta_m$ and each leftward step
being $\delta_\ell^{-1}$.  This word is
\[\omega_{(p,q)}=\delta_\ell^{p_{q-1}-p_q}\delta_m\;\cdots\;
\delta_\ell^{p_0-p_1}\delta_m\ .\]

Similar considerations give a braid description for the case when
$p>0$ and $q<0$. For $0\leq k\leq p$ put $q_k=\lceil kq/p \rceil$. The
braid is then
$\omega_{(p,q)}=\prod_{k=0}^{p-1}\delta_\ell\delta_m^{q_k-q_{k+1}}$.
\longpage

Assuming as before that $p,q\geq 2$,
we now use $\omega_{(p,q)}$ to compute the slope coefficients of the upper
tunnel of $K_{p,q}$. We have
\begin{gather*} \omega_{(p,q)}\sim \delta_m\delta_\ell^{p_{q-2}-p_{q-1}}\;
\cdots\;
\delta_m\delta_\ell^{p_1-p_2}\;\delta_m\delta_\ell^{p_0-p_1}\\
=\delta_m\sigma\;\sigma^{-1}\delta_\ell^{p_{q-2}-p_{q-1}}\;
\cdots\;
\delta_m \sigma\;\sigma^{-1}\delta_\ell^{p_1-p_2}\;
\delta_m \sigma\;\sigma^{-1}\delta_\ell^{p_0-p_1}\ .
\end{gather*}
\noindent
Putting $\omega_j=\delta_m\delta_\ell^{p_{j-1}-p_j}\;
\cdots\;\delta_m\delta_\ell^{p_1-p_2}\;\delta_m\delta_\ell^{-p_1}$,
we have
\[ t(\omega_j)=(p_j-p_{j-1}) + \cdots + (p_2-p_1) + (p_1-p_0) =
p_j-p_0=p_j\ .\]

Now, working from the right, Theorem~\ref{thm:winding} finds the
cabling slope sequence for the upper tunnel:
\begin{gather*}
[2(p_1-p_0),-1]=2p_1-1\\
[2t(\omega_1)+2(p_2-p_1),-1]=2p_2-1\\
\cdots\\
[2t(\omega_{q-2})+2(p_{q-1}-p_{q-2}),-1]= 2p_{q-1}-1
\end{gather*}
These give the trivial knot as long as $p_j\leq 1$, so we have reproved one
of the main results from~\cite{CMtorus}:
\begin{theorem}
Let $p$ and $q$ be relatively prime integers, both greater than $1$.  For
$1\leq k\leq q$, put $p_k=\lceil kp/q\rceil$, and let $k_0= \min\{ k\;|\;
p_k > 1\}$. Then the upper tunnel of $K_{p,q}$ is produced by $q-k_0$
cabling constructions, whose slopes are
\[[1/(2p_{k_0}-1)],\; 2p_{k_0+1}-1,\;\ldots\,, \;2p_{q-1}-1\ .\]\par
\label{thm:cabling_sequence}
\end{theorem}

\noindent Of course if $p,q<0$, then $K_{p,q}=K_{-p,-q}$. When $pq<0$,
there is an orientation-reversing equivalence from $K_{p,q}$ to $K_{p,-q}$
which takes upper tunnel to upper tunnel, so the slopes are just the
negatives of those given in Theorem~\ref{thm:cabling_sequence} for
$K_{|p|,|q|}$. The lower tunnel of $K_{p,q}$ is equivalent to the upper
tunnel of $K_{q,p}$, so Theorem~\ref{thm:cabling_sequence} also finds the
slope sequences of the lower tunnels.

\section{Toroidal $(1,1)$-positions}
\label{sec:toroidal}

As usual, we fix a decomposition $S^3=V\cup T\times I \cup W$, with
$T=T\times\{0\}=\partial W$ the standard torus in $S^3$. A knot is said to be
in a \textit{toroidal} position if it is contained in $T\times I$ and both
of the coordinate projections from $S^1\times S^1\times I$ to $S^1$
restrict to immersions on the knot. That is, when traveling along the knot,
neither of the $S^1$-coordinates ever reverses direction.

\begin{theorem}\label{thm:toroidal}
A simple or semisimple tunnel is the upper or lower tunnel of a toroidal
$(1,1)$-position if and only if its sequence of slope coefficients is of
the form
\[ [1/n_0], n_1, \ldots, n_k \]
with the $n_i, 1\leq i \leq k,$ either a nondecreasing sequence of positive odd integers or
a nonincreasing sequence of negative odd integers.
\end{theorem}
\noindent
Before proving Theorem~\ref{thm:toroidal}, we will use it to find the
toroidal $2$-bridge knots.
\begin{corollary}\label{coro:2bridge_toroidal}
A $2$-bridge knot $K$ admits a toroidal $(1,1)$-position if and only if it
satisfies one of the following equivalent conditions:
\begin{enumerate}
\item[(i)] For some $n>0$, its upper simple and upper semisimple
tunnels have respective slope sequences either $[1/(2n+1)]$ and
$[1/3],3,3,\ldots,3$, or $[2n/(2n+1)]$ and $[2/3],-3,-3,\ldots,-3$, where
the latter sequences in each case have length~$n$.
\item[(ii)] Its classifying invariants are $b/a=b'/a=1/(2n+1)$.
\item[(iii)] It is a torus knot, in fact a $(2n+1, \pm2)$-torus knot for some
$n\neq 0$.
\end{enumerate}
\end{corollary}
\begin{proof} Examining Proposition~\ref{prop:semisimple_slopes}, we find
that the only $2$-bridge knots whose upper semisimple tunnels have slope
sequences satisfying the condition of Theorem~\ref{thm:toroidal} are those
whose rational invariants are given by the continued fractions
$[-2,2,-2,2,\ldots,-2,2]$ and $[2,-2,2,-2,\ldots,2,-2]$, which give the
slope sequences in~(i) and correspond to the invariants in (ii). Using
Theorem~\ref{thm:cabling_sequence}, these are exactly the torus knots
listed in~(iii).
\end{proof}

We note that the $2$-bridge knots in Corollary~\ref{coro:2bridge_toroidal}
have only one $(1,1)$-position, so no two-bridge knot with two
$(1,1)$-positions is toroidal.

\begin{proof}[Proof of Theorem~\ref{thm:toroidal}]
A toroidal $(1,1)$-position is described by a braid of the form
$\delta_m^{a_1}\delta_\ell^{b_1}\cdots \delta_m^{a_n}\delta_\ell^{b_n}$
where the $a_i$ all have the same sign and the same is true of the
$b_i$'s. If the $a_i$ are all negative, apply an orientation-reversing
equivalence that reverses the orientation on the $S^1$-factor corresponding
to $\delta_m$ so that the $a_i$ are all positive. Since this negates all
the cabling slopes, it will not affect whether the slopes satisfy the
conclusion of the theorem.

To compute the cabling slopes, it is convenient to allow some $b_j=0$, and
rewrite the braid as
\[ \delta_m\delta_\ell^{b_0}\cdots \delta_m\delta_\ell^{b_k}
\sim
\delta_m\sigma\;\sigma^{-1}\delta_\ell^{b_0}\;\cdots\;
\delta_m\sigma\;\sigma^{-1}\;\delta_\ell^{b_k}\ .\]
\noindent Using Theorem~\ref{thm:winding} to read
off the cabling slopes, working from the right, we
obtain continued fractions
\[[-2b_k,-1],[-2(b_{k-1}+b_k),-1],\cdots,[-2(b_0+\cdots+b_k),-1]\]
and the slope invariants are as claimed.

Conversely, given the sequence $n_0,\ldots\,$, $n_k$, put
$m_j=-(n_j+1)/2$ and let $K$ be in the $(1,1)$-position
with braid description
\[\delta_m\delta_\ell^{m_k-m_{k-1}}\delta_m\delta_\ell^{m_{k-1}-m_{k-2}}
\cdots \delta_m\delta_\ell^{m_1-m_0}\delta_m\delta_\ell^{m_0}\ .\]
\noindent Calculation as above finds the slope coefficients of the
upper tunnel to be $[1/n_0]$, $n_1,\ldots\,$,
$n_k$.
\end{proof}

\section{Algorithmic computation of braid descriptions}
\label{sec:finding_braid_words}

Using Theorems~\ref{thm:unwinding} and~\ref{thm:winding}, it is not
difficult to obtain a braid description for a $(1,1)$-position of a
knot from the slope sequence $[m_0],m_1,\ldots\,$, $m_d$ of its upper
$(1,1)$-tunnel:
\begin{enumerate}
\item[(1)] Assuming that $m_0$ is selected so that $0<m_0<1$, write $1/m_0$
  as a continued fraction of the form $[2a_1,b_1,2a_2,\ldots,
    2a_n,b_n]$. Put $\omega_0=\delta_m\sigma \cdot
  \sigma^{b_n}\delta_\ell^{-a_n}\cdots \sigma^{b_1}\delta_\ell^{-a_1}$. If
  we start with the trivial knot in braid position with braid
  description $1$, then by
  Theorems~\ref{thm:unwinding} and~\ref{thm:winding} a cabling construction of slope
  $[2a_1,b_1,2a_2,\ldots, 2a_n,b_n]$ (on the upper tunnel) produces the
  knot with braid description $\omega_0$. Since this is the initial
  cabling, its simple slope is $[1/[2a_1,b_1,2a_2,\ldots,
      2a_n,b_n]]=[m_0]$.
\item[(2)] Write $m_1$ in the form $[2a_1,b_1,\ldots, 2a_n,b_n]$, and put
$\omega_1=\delta_m\sigma\cdot \sigma^{b_n}\delta_\ell^{-a_n}\cdots
\sigma^{b_1}\delta_\ell^{-a_1}\cdot \delta_\ell^{t(\omega_0)}$.
By Theorems~\ref{thm:unwinding} and~\ref{thm:winding}, a
cabling of slope
$m_1=[2t(\omega_0)+2(a_1-t(\omega_0)),b_1,2a_2,\ldots,b_n]$ now produces a
knot in $(1,1)$-position with
braid description $\omega_1\omega_0$.
\item[(3)] Write $m_2$ as
$[2a_1,b_1,\ldots, 2a_n,b_n]$, put $\omega_2
=\delta_m\sigma\cdot \sigma^{b_n}\delta_\ell^{-a_n}\cdots
\sigma^{b_n}\delta_\ell^{-a_1}\cdot \delta_\ell^{t(\omega_1\omega_0)}$, and so on.
\item[(4)] Put $w= \omega_d\cdots \omega_1\omega_0$.
\end{enumerate}

\section{Algorithmic computation of slope invariants}
\label{sec:word_simplify}

In this section we develop an effective algorithm for computing the slope
invariants of an upper or lower tunnel of a $(1,1)$-position given by a
braid description. We will only concern ourselves with the upper
tunnel, since the lower tunnel is the upper tunnel of the knot described by
the reverse braid.

The basic approach is obvious from the various examples that we have seen
computed; the main difficulties will arise in the technical matter of
dealing with anomalous infinite slopes.

We start by writing the given braid description in the form
\[ W_0(\delta_\ell,\sigma) \;\omega\; W_1(\delta_m,\sigma)\ ,\]
where $\omega$ starts with $\delta_m^{\pm1}$ and ends with
$\delta_\ell^{\pm1}$, and $W_0(\delta_\ell,\sigma)$ and
$W_1(\delta_m,\sigma)$ are words in the indicated letters.
Replacing each appearance of $\delta_m^{-1}$ in
$\omega$ with $\sigma\delta_m\sigma$, we can write
\[ \omega \sim \delta_m\sigma\cdot \omega_d(\delta_\ell,\sigma)\cdot
\delta_m\sigma\cdot \omega_{d-1}(\delta_\ell,\sigma) \cdots
\delta_m\sigma\cdot \omega_0(\delta_\ell,\sigma)
\]
where each $\omega_i(\delta_\ell,\sigma)$ lies in $\langle
\delta_\ell,\sigma\rangle$.

According to Theorem~\ref{thm:unwinding}, the $(1,1)$-position
described by $\omega$ is obtained starting from the trivial position (with
braid description $1$ and upper tunnel in standard position) by a sequence
of $d+1$ cabling constructions with slopes given as in
Theorem~\ref{thm:winding}. It may happen, however, that some have infinite
slope (hence, strictly speaking, are not genuine cabling
constructions). This occurs exactly when the slope given by
Theorem~\ref{thm:winding} would be infinite--- for instance, when
$\omega_i(\delta_\ell,\sigma)=\delta_\ell^k$ for some integer $k$, since
then the slope is of the form $[2t-2k,0]=\infty$.

To understand when the cabling produced by $\delta_m\sigma\cdot
\omega_i(\delta_\ell,\sigma)$ has infinite slope, we will need a
description of the subgroup $\langle \delta_\ell,\sigma\rangle$ of
$\mathcal{B}$. The Reidemeister-Schreier algorithm does not seem to be
effective in this case, but there is an easy argument giving a presentation
for this subgroup:
\begin{lemma}\label{lem:subgp_presentation}
The subgroup $\langle \delta_\ell,\sigma\rangle$ of
$\mathcal{B}$ has presentation
\[\langle
\delta_\ell,\sigma\;|\;(\delta_\ell\sigma)^2=1\rangle\ .\]
\end{lemma}
\begin{proof}
Let $\overline{\mathcal{B}}$ be the quotient of $\mathcal{B}$ obtained by
adding the relation $\delta_m^2=1$. It has presentation
\[ \overline{\mathcal{B}} = \langle \overline{\delta}_m,
\overline{\delta}_\ell, \overline{\sigma}\;|\;
(\overline{\delta}_\ell
\overline{\sigma})^2=1,
\overline{\delta}_m
\overline{\sigma}
\overline{\delta}_m^{-1} = \overline{\sigma}\,^{-1},
\overline{\delta}_m\,
\overline{\delta}_\ell\,
\overline{\delta}_m^{-1} = \overline{\sigma}^2
\overline{\delta}_\ell, \overline{\delta}_m^2=1
\rangle \]
\noindent which we may regard as a semidirect product
\[\langle \overline{\delta}_\ell,\overline{\sigma}\;|\;
(\overline{\delta}_\ell\overline{\sigma})^2 = 1 \rangle\rtimes
\langle \overline{\delta}_m\;|\; \overline{\delta}_m^2=1\rangle
\ .\]
There is an obvious homomorphism
$\langle \delta_\ell,\sigma\,|\,(\delta_\ell\sigma)^2=1\rangle \to \mathcal{B}$,
and the composition
$\langle \delta_\ell,\sigma\;|\;(\delta_\ell\sigma)^2=1\rangle \to \mathcal{B}\to
\overline{\mathcal{B}}$
carries
$\langle \delta_\ell,\sigma\;|\;(\delta_\ell\sigma)^2=1\rangle$
isomorphically to
$\langle\overline{\delta}_\ell,\overline{\sigma}\;|\;
(\overline{\delta}_\ell\overline{\sigma})^2 = 1 \rangle$. The lemma
follows.
\end{proof}

By Lemma~\ref{lem:subgp_presentation}, $\langle \delta_\ell,\sigma\rangle$
is a free product of the form $C_2*C_\infty$, where $C_2$ is generated by
$\delta_\ell\sigma$ and $C_\infty$ is generated by~$\sigma$. Recall the
elementary matrices $U=\begin{bmatrix} 1 & 1 \\ 0 & 1\end{bmatrix}$
and $L=\begin{bmatrix} 1 & 0 \\ 1 & 1\end{bmatrix}$ from
Section~\ref{sec:standard_tangles}.
\begin{lemma}\label{lem:PSL2subgroup}
The subgroup $\langle L^2,U\rangle$
of $\PSL(2,\Z)$ is given by the presentation
$\langle L^2,U\,|\, (L^{-2}U)^2=I\rangle$.
Consequently, sending $\delta_\ell$ to
$L^{-2}$ and $\sigma$ to $U$ defines an isomorphism from the
subgroup $\langle
\delta_\ell,\sigma\rangle$ of $\mathcal{B}$ to $\langle L^2,U\rangle$.
\end{lemma}
\begin{proof}
We use the homomorphism $\PSL(2,\Z)\to\PSL(2,\Z/2)=\SL(2,\Z/2)$, where the
latter is isomorphic to the permutation group on three letters. One can
check that $\langle L^2,U\rangle$ consists exactly of the elements of the
form $\begin{bmatrix}a & b\\ 2c &
  d\end{bmatrix}$, so
  $\langle L^2,U\rangle$ is the inverse image of the subgroup $\langle I,
\begin{bmatrix}1 & 1\\ 0 & 1\end{bmatrix}\rangle$ of $\PSL(2,\Z/2)$.
Therefore
$\langle L^2,U\rangle$ has index $3$ in $\PSL(2,\Z)$. Note also that this
shows that every element of $\langle L^2,U\rangle$ has even trace, and
hence is not of order $3$.

It is well-known that $\PSL(2,\Z)\cong C_2*C_3$. Since $\langle
L^2,U\rangle$ is a two-generator subgroup, it is a free product of two
cyclic subgroups. It contains the involution $L^{-2}U$, and no
elements of order $3$, so is isomorphic to either $C_2*C_\infty$ or
$C_2*C_2$. The latter is impossible since $C_2*C_2$ contains an infinite
cyclic subgroup of index $2$, which would have index $6$ in
$\PSL(2,\Z)$. Every element of order $2$ in $C_2* C_\infty$ is
conjugate to the generator of $C_2$, so as generators of the
free factors we may choose the involution
$L^{-2}U$ and the infinite order element $U$. The lemma
follows, making use of Lemma~\ref{lem:subgp_presentation}.
\end{proof}

\begin{lemma}\label{lem:cosets}
Let $S\subset \Q\cup\{\infty\}$ consist of the $a/b$ with $a$ odd. Sending
$\sigma^{b_n}\delta_\ell^{-a_n}\cdots \sigma^{b_1}\delta_\ell^{-a_1}$ to
the element of $\Q\cup\{\infty\}$ given by the continued fraction
$[2a_1,b_1,\ldots,2a_n,b_n]$ induces a bijection from the set of right
cosets $\langle \delta_\ell\rangle\backslash \langle
\delta_\ell,\sigma\rangle$ to~$S$.
\end{lemma}
\begin{proof}
Regard the elements $a/b$ of $\Q\cup\{\infty\}$ as row vectors
$\begin{bmatrix}a& b\end{bmatrix}$ (with $\begin{bmatrix}a& b\end{bmatrix}$
equivalent to $\begin{bmatrix}an& bn\end{bmatrix}$ for $n\neq 0$).
Define an action of
$\langle \delta_\ell,\sigma\rangle$ on the right on $S$ by
\begin{gather*}
\begin{bmatrix}a& b\end{bmatrix}\;\delta_\ell=
\begin{bmatrix}a& b\end{bmatrix}\;L^{-2}\\
\begin{bmatrix}a& b\end{bmatrix}\;\sigma=
\begin{bmatrix}a& b\end{bmatrix}\;U
\end{gather*}
where $U$ and $L$ are the upper and lower elementary matrices as in
Section~\ref{sec:standard_tangles}. Since $L^{-2}UL^{-2}U=-I$ acts
trivially on elements of $\Q\cup\{\infty\}$, Lemma~\ref{lem:PSL2subgroup}
shows that this is well-defined. We have
\[\begin{bmatrix}a& b\end{bmatrix}\;\sigma^{b_n}\delta_\ell^{-a_n}\cdots
\sigma^{b_1}\delta_\ell^{-a_1}=
\begin{bmatrix}a& b\end{bmatrix}\;U^{b_n}L^{2a_n}\cdots U^{b_1}L^{2a_1}\]
and taking the transpose gives
\[U^{2a_1}L^{b_1}\cdots U^{2a_n}L^{b_n}
\begin{bmatrix}a\\ b\end{bmatrix}
=\begin{bmatrix}q& s\\ p&r\end{bmatrix}
\begin{bmatrix}a\\ b\end{bmatrix}\ ,
\]
where, according to Lemma~14.3 of~\cite{CMtree}, $q/p$ has continued
fraction expansion $[2a_1,b_1,\ldots,2a_n,b_n]$. Every $a/b$ with $a$
odd can be written as a continued fraction of the form
$[2a_1,b_1,\ldots,2a_n,b_n]$ (see \cite{CMtree}[Lemma~14.2]),
so $\begin{bmatrix}1& 0\end{bmatrix}\;\sigma^{b_n}\delta_\ell^{-a_n}\cdots
\sigma^{b_1}\delta_\ell^{-a_1}=[2a_1,b_1,\ldots,2a_n,b_n]$
and therefore the action is transitive on $S$.
One can check easily that the stabilizer of
$\begin{bmatrix}1& 0\end{bmatrix}$ under the right action of
$\text{PSL}(2,\Z)$ is the subgroup generated by $L$. Using
Lemma~\ref{lem:PSL2subgroup}, the stabilizer of $1/0\in S$ under the
action of $\langle \delta_\ell,\sigma\rangle$ is $\langle
\delta_\ell\rangle$.
\end{proof}

\begin{proposition}\label{prop:infinite_slope}
Suppose that the cabling produced
by the segment $\delta_m\sigma\cdot \omega_i(\delta_\ell,\sigma)$ of the
braid description $\omega$ has infinite slope. Then
\[\omega_i(\delta_\ell,\sigma)=\delta_\ell^{-t(\omega_i(\delta_\ell,\sigma))}\]
in $\langle \delta_\ell,\sigma\rangle$.
\end{proposition}
\begin{proof}
Write $\omega_i(\delta_\ell,\sigma)=
\sigma^{b_n}\delta_\ell^{-a_n}\cdots \sigma^{b_1}\delta_\ell^{-a_1}$.  By
Theorem~\ref{thm:winding}, the cabling produced by $\delta_m\sigma\cdot
\omega_i(\delta_\ell,\sigma)$ has slope $[2t+2a_1,b_1,\ldots,2a_n,b_n]$,
where $t$ is the algebraic winding number of the portion of $\omega$
that follows $\omega_i(\delta_\ell,\sigma)$.  By Lemma~\ref{lem:cosets},
this is infinite exactly when $\sigma^{b_n}\delta_\ell^{-a_n}\cdots
\sigma^{b_1}\delta_\ell^{-a_1}$ is equal in $\langle
\delta_\ell,\sigma\rangle$ to some power $\delta_\ell^k$. Since
inserting or deleting the word $\delta_\ell\sigma\delta_\ell\sigma$ does
not change the winding number of a braid, we have
$t(\omega_i(\delta_\ell,\sigma))=t(\delta_\ell^k)=-k$ and the lemma
follows.
\end{proof}

We can now give the algorithm. Suppose that in the braid description
$\delta_m\sigma\cdot \omega_d(\delta_\ell,\sigma)\cdot
\delta_m\sigma\cdot \omega_{d-1}(\delta_\ell,\sigma) \cdots
\delta_m\sigma\cdot \omega_0(\delta_\ell,\sigma)$, the portion
$\delta_m\sigma\cdot \omega_i(\delta_\ell,\sigma)$ produces a cabling of
infinite slope. If $i\neq 0,d$, we have
\begin{eqnarray*}
& & \cdots \omega_{i+1}(\delta_\ell,\sigma)\;\delta_m\sigma\;
\omega_i(\delta_\ell,\sigma)\;\delta_m\sigma\;
\omega_{i-1}(\delta_\ell,\sigma)\;\delta_m\sigma \cdots\\
& = & \cdots \omega_{i+1}(\delta_\ell,\sigma)\;\delta_m\sigma\;
\delta_\ell^{-t(\omega_i(\delta_\ell,\sigma))}\;\delta_m\sigma \,
\omega_{i-1}(\delta_\ell,\sigma)\;\delta_m\sigma \cdots\\
& = & \cdots \omega_{i+1}(\delta_\ell,\sigma)\;
\delta_\ell^{t(\omega_i(\delta_\ell,\sigma))}
\;\delta_m\sigma\;
\delta_m\sigma\,\omega_{i-1}(\delta_\ell,\sigma)\;\delta_m\sigma  \cdots\\
& = & \cdots \omega_{i+1}(\delta_\ell,\sigma)\;
\delta_\ell^{t(\omega_i(\delta_\ell,\sigma))}
\,\omega_{i-1}(\delta_\ell,\sigma)\;\delta_m\sigma  \cdots
\end{eqnarray*}
with $d$ decreased by $2$. In going from the second line to the third,
we used the fact that
\[\delta_m\sigma \delta_\ell
= \delta_m \delta_\ell^{-1}\sigma^{-2}\sigma
= \delta_m \delta_\ell^{-1}
(\delta_\ell\delta_m^{-1}\delta_\ell^{-1}\delta_m)\sigma
= \delta_\ell^{-1}\delta_m\sigma\ .
\]

In the special case when $i=d$, this looks like
\begin{eqnarray*}
& & \delta_m\sigma\;
\delta_\ell^{-t(\omega_d(\delta_\ell,\sigma))}
\;\delta_m\sigma \;
\omega_{d-1}(\delta_\ell,\sigma)\; \delta_{d-2}\sigma\;\omega_3(\delta_\ell,\sigma)\cdots\\
& = &
\delta_\ell^{t(\omega_d(\delta_\ell,\sigma))}
\;\delta_m\sigma\;
\delta_m\sigma\;\omega_{d-1}(\delta_\ell,\sigma)\; \delta_{d-2}\sigma \;\omega_3(\delta_\ell,\sigma)\cdots\\
& = &
\delta_\ell^{t(\omega_d(\delta_\ell,\sigma))}
\;\omega_{d-1}(\delta_\ell,\sigma)\;
\delta_m\sigma \;\omega_{d-2}(\delta_\ell,\sigma)\cdots\\
& \sim &  \delta_m\sigma\;\omega_{d-2}(\delta_\ell,\sigma)
\end{eqnarray*}
with $d$ decreased by~$2$.

In the special case when $i=0$, we have
\begin{eqnarray*}
& & \cdots \omega_1(\delta_\ell,\sigma)\;\delta_m\sigma\;\omega_0(\delta_\ell,\sigma)\\
& =  & \cdots \omega_1(\delta_\ell,\sigma)\;\delta_m\sigma\;
\delta_\ell^{-t(\omega_0(\delta_\ell,\sigma))}\\
& = & \cdots \omega_1(\delta_\ell,\sigma)\;
\delta_\ell^{t(\omega_0(\delta_\ell,\sigma))}
\delta_m\sigma\\
& \sim & \cdots \omega_1(\delta_\ell,\sigma)\;
\delta_\ell^{t(\omega_0(\delta_\ell,\sigma))}
\end{eqnarray*}
with $d$ decreased by $1$.

We repeat these until there are no cablings of infinite slope. The cabling
slopes can then be read off from the new $\omega_i(\delta_\ell,\sigma)$,
starting from the rightmost. Some of the initial cablings may have integral
\textit{simple} slope, which occurs when their ordinary slope is of the
form~$1/k$. The first slope invariant is obtained by inverting the first
slope not of the form $1/k$ (and regarding the result as an element of
$\Q/\Z$). In terms of the algebraic manipulations we have been doing, what
is happening is this: When the slope associated to
$\omega_0(\delta_\ell,\sigma)$ is some $1/k$, its continued fraction has
value equal to that of the continued fraction
$[0,k]$. Lemma~\ref{lem:cosets} shows that $\omega_0(\delta_\ell,\sigma)$
is equal to $\delta_\ell^{-t(\omega_0(\delta_\ell,\sigma))}\sigma^k$.
So we have
\begin{eqnarray*}
& & \cdots \omega_1(\delta_\ell,\sigma)\;\delta_m\sigma\;\omega_0(\delta_\ell,\sigma)\\
& =  & \cdots
\omega_1(\delta_\ell,\sigma)\;\delta_m\sigma\;
\delta_\ell^{-t(\omega_0(\delta_\ell,\sigma))}\sigma^k\\
& = & \cdots
\omega_1(\delta_\ell,\sigma)\;
\delta_\ell^{t(\omega_0(\delta_\ell,\sigma))}
\,\delta_m\sigma\, \sigma^k\\
& \sim & \cdots
\omega_1(\delta_\ell,\sigma)\,\delta_\ell^{t(\omega_0(\delta_\ell,\sigma))}
\end{eqnarray*}
with $d$ decreased by $1$.

\section{Computations}
\label{sec:computation}

We have implemented the algorithms of
Sections~\ref{sec:finding_braid_words} and~\ref{sec:word_simplify} in a
script available at~\cite{slopes}. In this section, we give some sample
calculations.

Using the algorithm of Section~\ref{sec:word_simplify}, slopes of the upper
tunnel or lower tunnel are computed from a braid description. For
example, the braid
$\delta_m^3\sigma^{-2}\delta_\ell^3\sigma^{-4}\delta_m^{-1}\sigma^{-4}\delta_\ell^3$
gives
\smallskip

\noindent\texttt{Semisimple> upperSlopes( 'm 3 s -2 l 3 s -4 m -1 s -4 l 3' )}\\
\texttt{[ 21/25 ], 341/60, -13, -13}
\smallskip

\noindent To compute the lower slopes, the script just finds the reverse
braid and applies $\texttt{upperSlopes}$:
\smallskip

\noindent\texttt{Semisimple> lowerSlopes( 'm 3 s -2 l 3 s -4 m -1 s -4 l 3' )}\\
\texttt{[ 16/19 ], -7, -7, -195/31, -5, -5}
\smallskip

\noindent Using the method of Section~\ref{sec:finding_braid_words}, a
braid describing the $(1,1)$-position can be recovered from the upper
tunnel slope sequence. In the next example, the slope sequence
$[21/25]$, $341/60$, $-13$, $-13$ is represented as the input list
$[21,25,341,60,-13,1,-13,1]$:
\smallskip

\noindent\texttt{Semisimple> print braidWord( [21,25,341,60,-13,1,-13,1] )
)}\\
\texttt{m 3 s -3 m -1 l -2 m 1 l -1 s -4 m 1 s -4 m -1 l -2 m 1 l -1}
\smallskip

\noindent which checks:
\smallskip

\noindent\texttt{Semisimple> upperSlopes( 'm 3 s -3 m -1 l -2 m 1 l -1 s -4}\\
\texttt{m 1 s \allowbreak -4 m -1 l -2 m 1 l -1' )}\\
\texttt{[ 21/25 ], 341/60, -13, -13}
\smallskip

To compute the slopes of one tunnel associated to a $(1,1)$-position from
the slopes of the other, the script generates a braid describing an
upper tunnel which has those slopes, then find the slope sequence of the lower
tunnel:
\smallskip

\noindent\texttt{Semisimple> dualSlopes([21,25,341,60,-13,1,-13,1])}\\
\texttt{[ 16/19 ], -7, -7, -195/31, -5, -5}
\smallskip

\noindent\texttt{Semisimple> dualSlopes([16,19,-7,1,-7,1,-195,31,-5,1,-5,1])}\\
\texttt{[ 21/25 ], 341/60, -13, -13}
\smallskip

\noindent Functions are also included which calculate slope sequences for
semisimple tunnels of $2$-bridge and torus knots. For example, for
$2$-bridge knots we have\par
\smallskip

\noindent\texttt{Semisimple> twoBridge( 413, 227 )}\\
\noindent\texttt{Upper simple tunnel:     [ 131/413 ]}\\
\noindent\texttt{Upper semisimple tunnel: [ 1/3 ], 15/7, 9/5}\\
\noindent\texttt{Lower simple tunnel:     [ 227/413 ]}\\
\noindent\texttt{Lower semisimple tunnel: [ 2/5 ], -1, -3/2, 1, 1, 1, 3}
\smallskip

\noindent\texttt{Semisimple> print upperSemisimpleBraidWord( 413, 227 )}\\
\noindent\texttt{m -1 s -6 m -1 s 6 m -1 s 1 l -1}
\smallskip

\noindent\texttt{Semisimple> print lowerSimpleBraidWord( 413, 227 )}\\
\noindent\texttt{m -1 s 1 l -1 s 6 l -1 s -6 l -1}

\noindent For torus knots, we have:\par
\smallskip

\noindent\texttt{Semisimple> torusUpperSlopes( 13, 5 )}\\
\noindent\texttt{[ 1/5 ], 11, 15, 21}
\smallskip

\noindent\texttt{Semisimple> torusLowerSlopes( 13, 5 )}\\
\noindent\texttt{[ 1/3 ], 3, 3, 5, 5, 7, 7, 7, 9, 9}
\smallskip

\noindent\texttt{Semisimple> print fullTorusBraidWord( 13, 5 )}\\
\noindent\texttt{l -2 m 1 l -3 m 1 l -2 m 1 l -3 m 1 l -3 m 1}
\smallskip

Theorem~\ref{thm:slope_sequence_characterization} allows us to test whether
a slope sequence belongs to some $2$-bridge knot tunnel:\par
\smallskip

\noindent\texttt{Semisimple> find2BridgeKnot( [ 1, 3, 15, 7, 9, 5 ] )}\\
\noindent\texttt{The tunnel is the upper semisimple tunnel of K( 413, 227 ), or}\\
\noindent\texttt{equivalently the lower semisimple tunnel of K( 413, 131).}
\smallskip

\noindent\texttt{Semisimple> find2BridgeKnot( [ 1, 3, 15, 8, -9, 5 ] )}\\
\noindent\texttt{The tunnel is the upper semisimple tunnel of K( 493, 222 ), or}\\
\noindent\texttt{equivalently the lower semisimple tunnel of K( 493, 171).}
\smallskip

\noindent\texttt{Semisimple> find2BridgeKnot( [ 1, 3, 15, 11, 9, 5 ] )}\\
\noindent\texttt{Slopes other than first must be of the form 2 + 1/k or}\\
\noindent\texttt{2 - 1/k.}\par
\smallskip

\noindent\texttt{Semisimple> find2BridgeKnot( [ 1, 3, 15, 8, 9, 5 ] )}\\
\noindent\texttt{The ith and (i+1)st slopes must have opposite signs}\\
\noindent\texttt{when k sub i is even.}
\smallskip

\noindent\texttt{Semisimple> find2BridgeKnot( [ 1, 3, -15, 8, 9, 5 ] )}\\
\noindent\texttt{m1 must be positive or negative according as n0 is odd or}\\
\noindent\texttt{even.}\par

\bibliographystyle{amsplain}

\end{document}